\documentclass[11pt,oneside]{amsart}

\usepackage[T1]{fontenc} 
\usepackage[utf8]{inputenc}  
\usepackage[english]{babel}
\usepackage{tikz}
\usepackage{vaucanson-g}
\allowdisplaybreaks

\usepackage{vmargin}
\setmarginsrb{3cm}{2cm}{3cm}{2cm}{0cm}{2cm}{0cm}{1cm}

\theoremstyle{plain}
\newtheorem{theorem}{Theorem}
\newtheorem{lemma}[theorem]{Lemma}
\newtheorem{corollary}[theorem]{Corollary}
\newtheorem{proposition}[theorem]{Proposition}
\theoremstyle{definition}

\newtheorem{example}[theorem]{Example}
 
\numberwithin{equation}{section}

\newcommand{\R}{\mathbb R}
\newcommand{\N}{\mathbb N}
\newcommand{\rep}{\mathrm{rep}}
\newcommand{\val}{\mathrm{val}}
\newcommand{\lex}{\mathrm{lex}}
\newcommand{\Fac}{\mathrm{Fac}}
\newcommand{\Pref}{\mathrm{Pref}}
\newcommand{\A}{\mathcal{A}}
\newcommand{\B}{\boldsymbol{\beta}}
\newcommand{\Aalpha}{\boldsymbol{\alpha}}
\newcommand{\floor}[1]{\left\lfloor#1\right\rfloor}
\newcommand{\DB}{d_{\boldsymbol{\beta}}}
\newcommand{\DA}{d_{\boldsymbol{\alpha}}}
\newcommand{\DBi}[1]{d_{{\B}^{(#1)}}}
\newcommand{\qDB}{d_{\boldsymbol{\beta}}^{*}}
\newcommand{\qDBi}[1]{d_{\boldsymbol{\beta}^{(#1)}}^{*}}
\newcommand{\Int}{[\![0,p-1]\!]}
\DeclareMathOperator{\lcm}{lcm}
\DeclareMathOperator{\Card}{Card}

\title{Expansions in Cantor real bases}
\author{Émilie Charlier, Célia Cisternino}

\begin{document}

\begin{abstract}
We introduce and study series expansions of real numbers with an arbitrary Cantor real base $\B=(\beta_n)_{n\in\N}$, which we call $\B$-representations. In doing so, we generalize both representations of real numbers in real bases and through Cantor series. We show fundamental properties of $\B$-representations, each of which extends existing results on representations in a real base. In particular, we prove a generalization of Parry's theorem characterizing sequences of nonnegative integers that are the greedy $\B$-representations of some real number in the interval $[0,1)$. We pay special attention to periodic Cantor real bases, which we call alternate bases. In this case, we show that the $\B$-shift is sofic if and only if all quasi-greedy $\B^{(i)}$-expansions of $1$ are ultimately periodic, where $\B^{(i)}$ is the $i$-th shift of the Cantor real base $\B$.
\end{abstract}

\maketitle

\section{Introduction}

Cantor expansions of real numbers were originally introduced by Cantor in 1869~\cite{Cantor:1869}. A real number $x\in[0,1)$ is represented via a base sequence $(b_n)_{n\in\N}$ of integers greater than or equal to $2$ as follows:
\begin{equation}
\label{eq:CantorSeries}
x=\sum_{n=0}^{+\infty}\frac{a_n}{\prod_{i=0}^n b_i}
\end{equation}
where for each $n\in\N$, the digit $a_n$ belongs to the integer interval $[\![0,b_n-1]\!]$. If infinitely many digits $a_n$ are nonzero, then the series~\eqref{eq:CantorSeries} is called the Cantor series of $x$. Many studies are devoted to Cantor series, a large amount of which are concerned with the digit frequencies; see~\cite{Erdos&Renyi:1959,Galambos:1976,Kirschenhofer&Tichy:1984,Renyi:1956} to cite just a few. 

Representations of real numbers using a real base were first defined by Rényi in 1957~\cite{Renyi:1957}. In this context, a real number $x\in[0,1)$ is represented via a real base $\beta$ greater than~$1$ as follows:
\begin{equation}
\label{eq:RealBaseRep}
	x=\sum_{n=0}^{+\infty}\frac{a_{n}}{\beta^{n+1}}
\end{equation}
where the digits $a_n$ can be chosen by using several appropriate algorithms. The most commonly used algorithm is the greedy algorithm according to which for each $n\in\N$, $a_n=\lfloor \beta \,{T_\beta}^n(x)\rfloor$ where $T_\beta\colon [0,1)\mapsto [0,1),\ x\mapsto \beta x-\floor{\beta x}$. Expansions in a real base are extensively studied and we can only cite a few of the many possible references~\cite{Bertrand-Mathis:1986,Lothaire:2002,Parry:1960,Schmidt:1980}. 

This paper investigates series expansions of real numbers that are based on a sequence $\B=(\beta_n)_{n\in\N}$ of real numbers greater than~$1$. We call such a base sequence $\B$ a Cantor real base, and we talk about $\B$-representations. In doing so, we generalize both representations of real numbers through Cantor series and real base representations of real numbers. 

This paper has the following organization. In Section~\ref{sec:CantorRealBases}, we introduce the basic definitions and we give a characterization of those infinite words $a$ over the alphabet $\R_{\ge 0}$ for which there exists a Cantor real base $\B$ such that $\val_{\B}(a)=1$. In Section~\ref{sec:GreedyAlgorithm}, we define the greedy $\B$-representations of real numbers, which we call the $\B$-expansions. Then we prove several fundamental properties of $\B$-representations, each of which extends existing results on real base representations. 
In Section~\ref{sec:QuasiGreedyExpansions}, we introduce the quasi-greedy $\B$-expansion $\qDB(1)$ of $1$ and show that $\qDB(1)$ is the lexicographically greatest $\B$-representation not ending in $0^\omega$ of all real numbers in $[0,1]$.
In Section~\ref{sec:AdmissibleSequences}, we prove a generalization of Parry's theorem~\cite{Parry:1960} characterizing those infinite words over $\N$ that are the greedy $\B$-representations of some real number in the interval $[0,1)$.
In Section~\ref{sec:BShift}, we introduce the notion of $\B$-shift. We are able to give a description of the $\B$-shift in full generality.
In Section~\ref{sec:AlternateBases}, which is the last and biggest section, we focus on the periodic Cantor real bases, which we call alternate bases. We first give a characterization of those infinite words $a$ over the alphabet $\R_{\ge 0}$ for which there exists an alternate base $\B$ such that $\val_{\B}(a)=1$. Then we obtain a characterization of the $\B$-expansion of 1 among all $\B$-representations of $1$, which generalizes a result of Parry~\cite{Parry:1960}. Finally, generalizing Bertrand-Mathis' theorem~\cite{Bertrand-Mathis:1986}, we show that for any alternate base $\B$, the $\B$-shift is sofic if and only if all quasi-greedy $\B^{(i)}$-expansions of $1$ are ultimately periodic, where $\B^{(i)}$ is the $i$-th shift of the Cantor real base $\B$.

\section{Cantor real bases}
\label{sec:CantorRealBases}

Let $\B=(\beta_n)_{n\in\N}$ be a sequence of real numbers greater than $1$ such that $\prod_{n\in\N}\beta_n=+\infty$. We call such a sequence $\B$ a \emph{Cantor real base}, or simply a \emph{Cantor base}. We define the \emph{$\B$-value (partial) map} $\val_{\B}\colon (\R_{\ge 0})^\N\to \R_{\ge 0}$ by
\begin{equation}
\label{eq:representationCantor}
\val_{\B}(a)=\sum_{n\in\N}\frac{a_n}{\prod_{i=0}^n\beta_i}
\end{equation}
for any infinite word $a=a_0a_1a_2\cdots$ over $\R_{\ge 0}$, provided that the series converges. A \emph{$\B$-representation} of a nonnegative real number $x$ is an infinite word $a\in\N^\N$ such that $\val_{\B}(a)=x$. In particular, if $\B=(\beta,\beta,\ldots)$, then for all $x\in [0,1]$, a $\B$-representation of $x$ is a $\beta$-representation of $x$ as defined by Rényi~\cite{Renyi:1957}. 
In this case, we do not distinguish the notation $\B$ and $\beta$: we write $\val_{\beta}$ and we talk about $\beta$-representations, as usual. Also, any sequence $\B=(\beta_n)_{n\in\N}$ of real numbers greater than $1$ that takes only finitely many values is a Cantor base since in this case, the condition $\prod_{n\in\N}\beta_n=+\infty$ is trivially satisfied.

We will need to represent real numbers not only in a fixed Cantor base $\B$ but also in all Cantor bases obtained by shifting $\B$. We define
\[
	\B^{(n)}=(\beta_n,\beta_{n+1},\ldots)\quad \text{for all }n\in\N.
\]
In particular $\B^{(0)}=\B$. We will also need to consider shifted infinite words. Let us denote by $\sigma_A$ the \emph{shift operator}.
\[
	\sigma_A\colon A^\N\to A^\N,\ a_0a_1a_2\cdots\mapsto a_1a_2a_3\cdots
\] 
over the alphabet $A$. Whenever there is no ambiguity on the alphabet, we simply denote the shift operator by $\sigma$. Throughout this text, if $a$ is an infinite word then for all $n\in\N$, $a_n$ designates its letter indexed by $n$, so that $a=a_0a_1a_2\cdots$.

The $\B$-representations of $1$ will be of interest in what follows, in particular the greedy and the quasi-greedy expansions (see Sections~\ref{sec:GreedyAlgorithm} and~\ref{sec:QuasiGreedyExpansions}). We start our study by providing a characterization of those  infinite words $a$ over the alphabet $\R_{\ge 0}$ for which there exists a Cantor real base $\B$ such that $\val_{\B}(a)=1$.

When $\B=(\beta,\beta,\ldots)$, for any infinite word $a$ over $\N$ satisfying some suitable conditions, the equation $\val_{\beta}(a)=1$ admits a unique solution $\beta>1$ (see~\cite[Corollary 7.2.10]{Lothaire:2002}). This classical result remains true for nonnegative real digits and weaker conditions on the infinite word $a$.

\begin{lemma}
\label{lem:ExistRepresentation1}
Let $a$ be an infinite word over $\R_{\ge 0}$ such that $a_n\in O(n^d)$ for some $d\in\N$. There exists a real base $\beta$ such that $\val_{\beta}(a)=1$ if and only if $\sum_{n\in\N}a_n> 1$, in which case $\beta$ is unique and $\beta\ge a_0$, and if moreover for all $n\in\N$, $a_n \le a_0$, then $\beta\le a_0+1$.
\end{lemma} 

\begin{proof}
If $\sum_{n\in\N}a_n\le 1$ then for all real bases $\beta$, $\val_{\beta}(a)< 1$. Indeed, this is obvious if $a= 0^\omega$, and else $\val_{\beta}(a)< \sum_{n\in\N}a_n\le 1$.

Now, suppose that $\sum_{n\in\N}a_n> 1$. Let $N\in \N$ be such that $\sum_{n=0}^{N}a_n> 1$. The function $f\colon[0,1)\to \R,\ x\mapsto \sum_{n\in\N}a_nx^{n+1}$ is well-defined, continuous, increasing and such that $f(0)=0$ and that for all $x\in[0,1)$, $f(x)\ge \sum_{n=0}^{N}a_nx^{n+1}$. The function $g\colon\R\to \R,\ x\mapsto \sum_{n=0}^{N}a_nx^{n+1}$ is continuous, increasing and such that $g(0)=0$ and $g(1)>1$. Therefore, there exists a unique $x_0\in(0,1)$ such that $g(x_0)=1$, and hence such that $f(x_0)\ge 1$. Now, there exists a unique $\gamma\in(0,x_0]$ such that $f(\gamma)=1$. By setting $\beta=\frac{1}{\gamma}$, we get that $\beta\ge \frac{1}{x_0}>1$ and $\val_{\beta}(a)=f\left(\frac{1}{\beta}\right)=1$. Moreover, $\beta\ge a_0$ for otherwise $f\left(\frac{1}{\beta}\right)>f\left(\frac{1}{a_0}\right)\ge 1$. 

If moreover for all $n\in\N$, $a_n \le a_0$, then $\beta\le a_0+1$ for otherwise we would have
\[	
	\val_{\beta}(a)
	=\sum_{n\in\N} \frac{a_n}{\beta^{n+1}}
	 < a_0 \sum_{n\in\N}\frac{1}{(a_0+1)^{n+1}}
	 =1.
\]
\end{proof}

No upper bound on the growth order of the digits $a_n$ is needed in order to find a Cantor base $\B$ such that $\val_{\B}(a)=1$.

\begin{lemma}
\label{lem:ExistRepresentation2}
Let $a$ be an infinite word over $\R_{\ge 0}$ such that $\sum_{n\in\N}a_n=+\infty$. Then there exists a Cantor base $\B$ such that $\val_{\B}(a)=1$.
\end{lemma}

\begin{proof}%preuve2
First of all, observe that the hypothesis implies that $a$ does not end in $0^\omega$ and that $\prod_{n\in\N}(a_n+1)=+\infty$. 

We define two sequences of nonnegative integers $(n_k)_{1\le k\le K}$ and $(\ell_k)_{1\le k\le K}$ where $K\in\N\cup\{+\infty\}$. 
The length $K$ of these two sequences is the number of zero blocks in $a$, i.e.\ the factors of the form $0^\ell$ which are neither preceded nor followed by $0$ in $a$. Two cases stand out: either $K\in\N$ or $K=+\infty$. We describe the two cases at once. In order to do so, it should be understood that the parts of the definition where $k> K$ should just be ignored when $K\in\N$. Let $n_1$ denote the least $n\in\N$ such that $a_n=0$ and let $\ell_1$ denote the least $\ell\in\N$ such that $a_{n_1+\ell}>0$. Then for $k\ge 2$, let $n_k$ denote the least integer $n> n_{k-1}+\ell_{k-1}$ such that $a_n=0$ and let $\ell_k$ denote the least $\ell\in\N$ such that $a_{n_k+\ell}>0$. Thus, $(n_k)_{1\le k\le K}$ is the sequence of positions of appearance of the successive zero blocks in $a$ and $(\ell_k)_{1\le k\le K}$ is the sequence of lengths of these blocks. 

Next, for all $k\in[\![1,K]\!]$, we pick any $\alpha_k$ in the interval $(1,\sqrt[\ell_k]{a_{n_k+\ell_k}+1})$. For all $n\in\N$, we define
\[
	\beta_n=
	\begin{cases}
	a_n+1	& \text{if }n\in [\![0,n_1-1]\!]\text{ or }n\in
					\bigcup_{k=1}^K[\![n_k+\ell_k+1,n_{k+1}-1]\!]\\
	\alpha_k & \text{if }n\in[\![n_k,n_k+\ell_k-1]\!] \text{ for some }k\in[\![1,K]\!]\\			
	\frac{a_n+1}{\alpha_k^{\ell_k}}	
			&\text{if }n=n_k+\ell_k \text{ for some }k\in[\![1,K]\!]	
	\end{cases}
\]
where we set $n_{K+1}=+\infty$ if $K\in\N$. In particular if $K=0$, i.e.\ if for all $n\in\N$, $a_n>0$, then for all $n\in\N$, $\beta_n=a_n+1$.

Let us show that in any case, the obtained sequence $\B=(\beta_n)_{n\in\N}$ is such that $\prod_{n\in \N}\beta_n=+\infty$ and $\val_{\B}(a)=1$. By construction, 
\[
	\prod_{n\in\N}\beta_n
	=\prod_{n=0}^{n_1-1}(a_n+1)\cdot 
	\prod_{k=1}^K
	\left(\alpha_k^{\ell_k}\cdot
	\frac{a_{n_k+\ell_k}+1}{\alpha_k^{\ell_k}}
	\cdot
	\prod_{n=n_k+\ell_k+1}^{n_{k+1}-1}(a_n+1) 
	\right)
	=\prod_{n\in\N}(a_n+1).
\]
By induction we can show that
\[
	\sum_{n=0}^{n_k+\ell_k}\frac{a_n}{\prod_{i=0}^n\beta_i}
	= 1-\frac{1}{\prod_{i=0}^{n_k+\ell_k}\beta_i}
	\quad\text{for all }k\in[\![1,K]\!].
\]
If $K=+\infty$ then we obtain that $\val_{\B}(a)=1$ by letting $k$ tend to infinity. Otherwise, $K\in\N$. Set $n_0=-1$ and $\ell_0=0$. By induction again, we can show that
\[
	\sum_{n=n_K+\ell_K+1}^m
	\frac{a_n}{\prod_{i=n_K+\ell_K+1}^n\beta_i}
	=1-\frac{1}{\prod_{i=n_K+\ell_K+1}^m\beta_i}
	\quad\text{for all }m\in\N.
\]
By letting $m$ tend to infinity, we get
\[
	\val_{\B^{(n_K+\ell_K+1)}}(\sigma^{n_K+\ell_K+1}(a))=1.
\]
Finally, we obtain
\begin{align*}
	\val_{\B}(a)
	&=\sum_{n=0}^{n_K+\ell_K}
	\frac{a_n}{\prod_{i=0}^n\beta_i}+
	\sum_{n=n_K+\ell_K+1}^{+\infty}
	\frac{a_n}{\prod_{i=0}^n\beta_i}\\
	&=1-\frac{1}{\prod_{i=0}^{n_K+\ell_K}\beta_i}		
	+\frac{\val_{\B^{(n_K+\ell_K+1)}}(\sigma^{n_K+\ell_K+1}(a))}{\prod_{i=0}^{n_K+\ell_K}\beta_i}\\
	&= 1.
\end{align*}
\end{proof}

\begin{proposition}
\label{prop:ExistRepresentation}
Let $a$ be an infinite word over $\R_{\ge 0}$. There exists a Cantor base $\B$ such that $\val_{\B}(a)=1$ if and only if %either $a$ does not end in $0^\omega$ or 
$\sum_{n\in\N}a_n> 1$. 
\end{proposition}

\begin{proof}
Similarly as in the proof of Lemma~\ref{lem:ExistRepresentation1}, the condition $\sum_{n\in\N}a_n> 1$ is necessary. Now, suppose that $\sum_{n\in\N}a_n> 1$. If $\sum_{n\in\N}a_n=+\infty$ then we use Lemma~\ref{lem:ExistRepresentation2}. Otherwise, we have $1<\sum_{n\in\N}a_n<+\infty$ and we apply Lemma~\ref{lem:ExistRepresentation1}.
\end{proof}

\section{The greedy algorithm}
\label{sec:GreedyAlgorithm}

For $x\in[0,1]$, a distinguished $\B$-representation $\varepsilon_0(x)\varepsilon_1(x)\varepsilon_2(x)\cdots$ is given thanks to the \emph{greedy algorithm}:
\begin{itemize}
\item $\varepsilon_{0}(x)=\floor{\beta_0 x}$ and $r_0(x)=\beta_0 x-\varepsilon_0(x)$
\item $\varepsilon_{n}(x)=\floor{\beta_n r_{n-1}(x)}$ and $r_n=\beta_n r_{n-1}(x)-\varepsilon_n(x)$ for $n\in\N_{\ge 1}$.
\end{itemize}
The obtained $\B$-representation of $x$ is denoted by $\DB(x)$ and is called the \emph{$\B$-expansion} of $x$. Note that the $n$-th digit $\varepsilon_{n}(x)$ belongs to $\{0,\ldots,\floor{\beta_n}\}$. We let $A_{\B}$ denote the (possibly infinite) alphabet $\{0,\ldots,\sup_{n\in\N}\floor{\beta_n}\}$. The algorithm is called greedy since at each step it chooses the largest possible digit. Indeed, consider $x\in [0,1]$ and $\ell\in\N$, and suppose that the digits $\varepsilon_0(x),\ldots ,\varepsilon_{\ell-1}(x)$ are already known. Then the digit $\varepsilon_\ell(x)$ is the largest element of $ [\![0,\floor{\beta_\ell}]\!]$ such that $\sum_{n=0}^\ell \frac{\varepsilon_n(x)}{\prod_{i=0}^n \beta_i}\le x$. Thus
\[
	x=\sum_{n=0}^\ell\frac{\varepsilon_n(x)}{\prod_{i=0}^n\beta_i}
	+\frac{r_\ell(x)}{\prod_{i=0}^\ell\beta_i}
\]
where $r_\ell(x)\in[0,1)$. Note that since a Cantor base satisfies $\prod_{n\in\N}\beta_n=+\infty$, the latter equality implies the convergence of the greedy algorithm and that $x=\val_{\B}(\DB(x))$. We let $D_{\B}$ denote the subset of $A_{\B}^\N$ of all $\B$-expansions of real numbers in the interval $[0,1)$:
\[
	D_{\B}=\{\DB(x)\colon x\in[0,1)\}.
\]  
In what follows, the $\B$-expansion of $1$ will play a special role. For the sake of clarity, we denote its digits by $\varepsilon_{n}$ instead of $\varepsilon_{n}(1)$. We sometimes write $\varepsilon_{\B,n}(x)$ and $\varepsilon_{\B,n}$ instead of $\varepsilon_{n}(x)$ and $\varepsilon_{n}$ when the Cantor base $\B$ needs to be emphasized. As previously mentioned, if $\B=(\beta,\beta,\ldots)$, then for all $x\in [0,1]$, the $\B$-expansion of $x$ is equal to the usual $\beta$-expansion of $x$ as defined by Rényi~\cite{Renyi:1957} and we write indistinctly $\B$ or $\beta$.

We can also express the obtained digits $\varepsilon_{n}(x)$ and remainders $r_n(x)$ thanks to the $\beta_n$-transformations. For $\beta>1$, the $\beta$-transformation is the map
\[
	T_{\beta}\colon [0,1)\to [0,1),\ x \mapsto \beta x -\floor{\beta x}.
\]
Then for all $x\in[0,1)$ and $n\in\N$, we have
\[
	\varepsilon_{n}(x)
	=\floor{\beta_n\big(T_{\beta_{n-1}}\circ \cdots \circ T_{\beta_0}(x)\big)}
\quad \text{and}\quad
	r_{n}(x)
	=T_{\beta_n}\circ \cdots \circ T_{\beta_0}(x).
\]

\begin{example}
If there exists $n\in \N$ such that $\beta_n$ is an integer (without any restriction on the other $\beta_m$), then $d_{\B^{(n)}}(1)=\beta_n 0^\omega$. 
\end{example}

\begin{example}
For $n\in \N$, let $\alpha_n=1+\frac{1}{2^{n+1}}$ and $\beta_n=2+\frac{1}{2^{n+1}}$. The sequence $\Aalpha=(\alpha_n)_{n\in\N}$ is not a Cantor base since $\prod_{n\in\N}\alpha_n<+\infty$. If we perform the greedy algorithm on $x=1$ for the sequence $\Aalpha$, we obtain the sequence of digits $10^\omega$, which is clearly not an $\Aalpha$-representation of $1$. However, the sequence $\B=(\beta_n)_{n\in\N}$ is indeed a Cantor base since $\prod_{n\in\N}\beta_n=+\infty$. 
\end{example}

\begin{example}
Let $\alpha=\frac{1+ \sqrt{13}}{2}$ and $\beta=\frac{5+ \sqrt{13}}{6}$.
\begin{enumerate}
\item Consider $\B=(\beta_n)_{n\in\N}$ the Cantor base defined by 
\[
\beta_n=\begin{cases} 
\alpha & \text{if } | \rep_2(n) |_1 \equiv 0 \pmod 2\\
\beta & \text{otherwise}
\end{cases}
\]
for all $n\in \N$, where $\rep_2$ is the function mapping any nonnegative integer to its $2$-expansion. We get $\B=(\alpha,\beta,\beta,\alpha,\beta,\alpha,\alpha,\beta, \ldots)$ where the infinite word $\beta_0\beta_1\beta_2\cdots$ is the Thue-Morse word over the alphabet $\{\alpha,\beta\}$.
We compute $\DB(1)=20010110^{\omega}$, $\DBi{1}(1)=1010110^\omega$ and $\DBi{2}(1)=110^{\omega}$.
\item Consider $\B=(\sqrt{13},\alpha,\beta,\alpha,\beta,\alpha,\beta,\ldots)$. 
It is easily checked that $\DB(1)=3(10)^\omega$ and that for all $m\in\N$, $\DBi{2m+1}(1)=2010^\omega$ and $\DBi{2m+2}(1)=110^\omega$.
\end{enumerate}
\end{example}

We call an \emph{alternate base} a periodic Cantor base, i.e.\ a Cantor base for which there exists $p\in\N_{\ge 1}$ such that for all $n\in\N$, $\beta_n=\beta_{n+p}$. In this case we simply note $\B=(\overline{\beta_0,\ldots,\beta_{p-1}})$ and the integer $p$ is called the \emph{length} of the alternate base $\B$.  In what follows, most examples will be alternate bases and Section~\ref{sec:AlternateBases} will be specifically devoted to their study. 

\begin{example} 
\label{ex:3phiphi}
Let $\varphi=\frac{1+\sqrt{5}}{2}$ be the Golden Ratio and let $\B=(\overline{3,\varphi,\varphi})$. 
For all $m\in\N$,  we have $\DBi{3m}(1)=30^\omega$, $\DBi{3m+1}(1)=110^\omega$ and $\DBi{3m+2}(1)=1(110)^\omega$. 
\end{example}

Let us show that the classical properties of the $\beta$-expansion theory are still valid for Cantor bases. Some are just an adaptation of the related proofs in \cite{Lothaire:2002} but for the sake of completeness the details are written. From now on, unless otherwise stated, we consider a fixed Cantor base $\B=(\beta_n)_{n\in\N}$.

\begin{proposition}
For all $x\in [0,1)$ and all $n\in \N$, we have 
\[
	\sigma^n\circ \DB (x)
	= d_{\B^{(n)}}\circ T_{\beta_{n-1}}\circ \cdots\circ T_{\beta_0} (x).
\]
\end{proposition}

\begin{proof}
This is a straightforward verification.
\end{proof}

\begin{lemma}
\label{lem:Greedy}
For all infinite words $a$ over $\N$ and all $x\in[0,1]$, $a=\DB(x)$ if and only if $\val_{\B}(a)=x$ and for all $\ell\in\N$,
\begin{equation}
\label{eq:GreedyCondition}
\sum_{n=\ell+1}^{+\infty}\frac{a_n}{\prod_{i=0}^n{\beta_i}} < \frac{1}{\prod_{i=0}^{\ell}{\beta_i}}.
%\quad\text{for all } \ell\in\N.
\end{equation}
\end{lemma}

\begin{proof}
From the greedy algorithm, for all $x\in[0,1]$,
$\val_{\B}(\DB(x))=x$ and for all $\ell\in\N$,
\[
	\left(\sum_{n=\ell+1}^{+\infty}
	\frac{\varepsilon_n(x)}{\prod_{i=0}^n{\beta_i}}\right)
	\prod_{i=0}^{\ell}{\beta_i}
	=\left(x-\sum_{n=0}^{\ell}
	\frac{\varepsilon_n(x)}{\prod_{i=0}^n{\beta_i}}\right)
	\prod_{i=0}^{\ell}{\beta_i}
	=r_{\ell}(x)<1.
\]
Conversely, suppose that $a$ is an infinite word over $\N$ such that $\val_{\B}(a)=x$ and such that for all $\ell\in\N$, \eqref{eq:GreedyCondition} holds. Let us show by induction that for all $m\in\N$, $a_m=\varepsilon_m(x)$. From~\eqref{eq:GreedyCondition} for $\ell=0$, we get that $x-\frac{a_0}{\beta_0}<\frac{1}{\beta_0}$. Thus, $\beta_0 x-1<a_0$. Since $\frac{a_0}{\beta_0}\le x$, we get that $a_0\le \beta_0 x$. Therefore, $a_0=\floor{\beta_0 x}=\varepsilon_0(x)$. Now, suppose that $m\in\N_{\ge 1}$ and that for $n\in[\![0,m-1]\!]$, $a_n=\varepsilon_n(x)$. Then
\[
	a_m+	\left(\sum_{n=m+1}^{+\infty}
	\frac{a_n}{\prod_{i=0}^n{\beta_i}}\right)
	\prod_{i=0}^m{\beta_i}
	=\varepsilon_m(x)+r_m(x).	
\]
By using~\eqref{eq:GreedyCondition} for $\ell=m$ and since $r_m(x)<1$, we obtain that $a_m=\varepsilon_m(x)$.
\end{proof}

\begin{proposition}
\label{prop:ShiftedGreedy}
Let $a$ be a $\B$-representation of some real number $x$ in $[0,1]$. 
Then the following four assertions are equivalent.
\begin{enumerate}
\item The infinite word $a$ is the $\B$-expansion of $x$.
\item For all $n\in \N_{\ge 1}$, $\val_{\B^{(n)}}(\sigma^{n}(a))<1$.
\item The infinite word $\sigma(a)$ belongs to $D_{\B^{(1)}}$.
\item For all $n\in \N_{\ge 1}$, $\sigma^{n}(a)$ belongs to $D_{\B^{(n)}}$.
\end{enumerate}
\end{proposition}

\begin{proof}
Since $\val_{\B}(a)=x\in[0,1]$, it follows from Lemma~\ref{lem:Greedy} that $a=\DB(x)$ if and only if for all $\ell\in\N$, \eqref{eq:GreedyCondition} holds. 
In order to obtain the equivalences between the first three items, it suffices to note that the greedy condition~\eqref{eq:GreedyCondition} can be rewritten as $\val_{\B^{(\ell+1)}}(\sigma^{\ell+1}(a))<1$.
Clearly $(4)$ implies $(3)$. Finally we obtain that $(3)$ implies $(4)$ by iterating the implication $(1) \implies (3)$.
\end{proof}

\begin{corollary}
\label{cor:ShiftedGreedy}
An infinite word $a$ over $\N$ belongs to $D_{\B}$ if and only if for all $n\in\N$, $\val_{\B^{(n)}}(\sigma^n(a))<1$.
\end{corollary}

%\begin{proof}
%By Lemma~\ref{lem:Greedy}, $a$ belongs to $D_{\B}$ if and only if $\val_{\B}(a)<1$ and for all $\ell\in\N$, \eqref{eq:GreedyCondition} holds. We obtain the result by using the reformulation~\eqref{eq:GreedyConditionVal}.
%\end{proof}

\begin{proposition}
\label{pro:GreedyLexGreatest}
The $\B$-expansion of a real number $x\in [0,1]$ is lexicographically maximal among all $\B$-representations of $x$.
\end{proposition}

\begin{proof}
Let $x\in [0,1]$ and $a\in\N^\N$ be a $\B$-representation of $x$.
Proceed by contradiction and suppose that $a >_{\lex} \DB(x)$. There exists $\ell\in\N$ such that $\varepsilon_0(x)\cdots \varepsilon_{\ell-1}(x)=a_0\cdots a_{\ell-1}$ and $a_\ell>\varepsilon_\ell(x)$. Then
\[
	\sum_{n=\ell}^{+\infty}\frac{\varepsilon_n(x)}{\prod_{i=0}^n{\beta_i}} 
	 = \sum_{n=\ell}^{+\infty}\frac{a_n}{\prod_{i=0}^n{\beta_i}} 
	 \ge \frac{\varepsilon_\ell(x)+1}{\prod_{i=0}^\ell{\beta_i}}
		+ \sum_{n=\ell+1}^{+\infty}\frac{a_n}{\prod_{i=0}^n{\beta_i}}
\] 
and hence
\[ 
	\sum_{n=\ell+1}^{+\infty}\frac{\varepsilon_n(x)}{\prod_{i=0}^n{\beta_i}} \ge \frac{1}{\prod_{i=0}^{\ell}{\beta_i}}
\]
which is impossible by Lemma~\ref{lem:Greedy}.
\end{proof}

\begin{proposition}
\label{pro:Increasing}
The function $\DB\colon [0,1]\to {A_{\B}}^\N$ is increasing: 
\[
	\forall x,y \in [0,1],\quad x<y \iff \DB(x) <_{\lex} \DB(y).
\]
\end{proposition}

\begin{proof}
Suppose that $\DB(x)<_{\lex} \DB(y)$. There exists $\ell\in\N$ such that $\varepsilon_0(x)\cdots \varepsilon_{\ell-1}(x)=\varepsilon_0(y)\cdots \varepsilon_{\ell-1}(y)$ and $\varepsilon_\ell(x)<\varepsilon_\ell(y)$. By Lemma~\ref{lem:Greedy}, we get
\[
	x	 =\sum_{n\in\N}\frac{\varepsilon_n(x)}{\prod_{i=0}^n{\beta_i}} \\
		 < \sum_{n=0}^{\ell-1}\frac{\varepsilon_n(y)}{\prod_{i=0}^n{\beta_i}} 
	 		+ \frac{\varepsilon_\ell(y)-1}{\prod_{i=0}^\ell{\beta_i}}
			+ \frac{1}{\prod_{i=0}^\ell{\beta_i}} \\
		 = \sum_{n=0}^{\ell}\frac{\varepsilon_n(y)}{\prod_{i=0}^n{\beta_i}} \\
		 \le y.
\]
It follows immediately that $x<y$ implies $\DB(x) <_{\lex} \DB(y)$.
\end{proof}

\begin{corollary}
\label{cor:DB1LexGreatest}
If $a$ is an infinite word over $\N$ such that $\val_{\B}(a)\le 1$, then $a\le_{\lex} \DB(1)$. In particular, $\DB(1)$ is lexicographically maximal among all $\B$-representations of all real numbers in $[0,1]$.
\end{corollary}

\begin{proof}
Let $a$ be an infinite word over $\N$ such that $\val_{\B}(a)\le 1$. By Propositions~\ref{pro:GreedyLexGreatest} and~\ref{pro:Increasing}, $a\le_{\lex}\DB(\val_{\B}(a))\le_{\lex}\DB(1)$.
\end{proof}

Recall the property of the $\beta$-expansions stating that considering two bases $\alpha$ and $\beta$, $\alpha < \beta$ if and only if $d_\alpha(1)< d_\beta(1)$ \cite{Parry:1960}. The following proposition shows the generalization of a weaker version of this property in the case of Cantor bases.

\begin{proposition}
Let $\boldsymbol{\alpha}=(\alpha_n)_{n\in\N}$ and $\B=(\beta_n)_{n\in\N}$ be two Cantor bases such that for all $n\in\N$, $\prod_{i=0}^n\alpha_i \le \prod_{i=0}^n\beta_i$. Then for all $x \in [0,1]$, we have $\DA(x)
\le_{\lex} \DB(x)$.
\end{proposition}

\begin{proof}
Let $x \in [0,1]$ and suppose to the contrary that $\DA(x)>_{\lex}\DB(x)$. Thus, there exists $\ell\in\N$ such that $\varepsilon_{\Aalpha,0}(x)\cdots \varepsilon_{\Aalpha,\ell-1}(x)=\varepsilon_{\B,0}(x)\cdots \varepsilon_{\B,\ell-1}(x)$ and $\varepsilon_{\Aalpha,\ell}(x)>\varepsilon_{\B,\ell}(x)$. From Lemma~\ref{lem:Greedy} and from the hypothesis, 
we obtain that
\[
	x	
	\le \sum_{n=0}^{\ell-1}
		\frac{\varepsilon_{\Aalpha,n}(x)}{\prod_{i=0}^n{\beta_i}} 
		+\frac{\varepsilon_{\Aalpha,\ell}(x)-1}{\prod_{i=0}^\ell{\beta_i}}
		+\sum_{n=\ell+1}^{+\infty}				
		\frac{\varepsilon_{\B,n}(x)}{\prod_{i=0}^n{\beta_i}} 
	<\sum_{n=0}^{\ell}
		\frac{\varepsilon_{\Aalpha,n}(x)}{\prod_{i=0}^n{\beta_i}} 
	\le \sum_{n=0}^{\ell}
		\frac{\varepsilon_{\Aalpha,n}(x)}{\prod_{i=0}^n{\alpha_i}}
	\le x,
\]
a contradiction.
\end{proof}

\begin{corollary}
Let $\boldsymbol{\alpha}=(\alpha_n)_{n\in\N}$ and $\B=(\beta_n)_{n\in\N}$ be two Cantor bases such that  for all $n\in\N$, $\alpha_n \le\beta_n$.
Then for all $x \in [0,1]$, we have $\DA(x)\le_{\lex} \DB(x)$.
\end{corollary}

It is not true that $\DA(1)<_{\lex} \DB(1)$ implies that for all $n\in\N$, $\prod_{i=0}^{n}\alpha_i \le \prod_{i=0}^{n}\beta_i$ as the following example shows. The same example shows that the lexicographic order on the Cantor bases is not sufficient either. Here, the term lexicographic order refers to the following order: $\Aalpha<\B$ whenever there exists $\ell\in\N$ such that $\alpha_n=\beta_n$ for $n\in[\![0,\ell-1]\!]$ and $\alpha_\ell<\beta_\ell$.

\begin{example}
Let $\Aalpha=(\overline{2+\sqrt{3},2})$ and $\B=(\overline{2+\sqrt{2},5})$. Then $\DA(1)=31^\omega$ and $\DB(1)$ starts with the prefix $32$, hence $\DA(1)<_{\lex} \DB(1)$.
\end{example}

\section{Quasi-greedy expansions}
\label{sec:QuasiGreedyExpansions}

A $\B$-representation is said to be \emph{finite} if it ends with infinitely many zeros, and \emph{infinite} otherwise. The \emph{length} of a finite $\B$-representation is the length of the longest prefix ending in a non-zero digit. When a $\B$-representation is finite, we usually omit to write the tail of zeros.
 
When the $\B$-expansion of $1$ is finite, we show how to modify it in order to obtain an infinite $\B$-representation of $1$ that is lexicographically maximal among all infinite $\B$-representations of $1$. The obtained $\B$-representation is denoted by $\qDB(1)$ and is called the \emph{quasi-greedy $\B$-expansion} of $1$. It is defined recursively as follows:
\begin{align}
\label{def:quasigreedy}
	\DB^*(1)
	=\begin{cases}
		\DB(1) 	&\text{if } \DB(1) \text{ is infinite} \\
		\varepsilon_0\cdots \varepsilon_{\ell-2}(\varepsilon_{\ell-1} -1)d_{\B^{(\ell)}}^{*}(1) 		
		&\text{if } \DB(1)=\varepsilon_0\cdots \varepsilon_{\ell-1} \text{ with } \ell \in \N_{\ge 1},\ \varepsilon_{\ell-1} >0.
	\end{cases}
\end{align}

\begin{example}\label{ex2:3phiphi}
Let $\B=(\overline{3,\varphi,\varphi})$ the alternate base already considered in Example~\ref{ex:3phiphi}. Then we directly have that for all $m\in\N$, $\qDBi{3m+2}(1)=\DBi{3m+2}(1)=1(110)^\omega$. In order to compute $\qDBi{3m}(1)$ and $\qDBi{3m+1}(1)$, we need to go through the definition several times. For all $m\in\N$, we compute $\qDBi{3m}(1)=2\qDBi{3m+1}(1)=210\qDBi{3m+3}(1)=210\qDBi{3m}(1)=(210)^\omega$ and $\qDBi{3m+1}(1) =10\qDBi{3m+3}(1)=10(210)^\omega=(102)^\omega$. 
\end{example}

\begin{example}
Let $\B=(\overline{\beta_0,\ldots,\beta_{p-1}})$ be an alternate base such that for all $i\in\Int$, $\beta_i\in\N_{\ge 2}$. Then for all $i\in\Int$, $\DBi{i}(1)=\beta_i0^\omega$ and 
\[
	\qDBi{i}(1)=((\beta_i-1)\cdots(\beta_{p-1}-1)(\beta_0-1)\ldots(\beta_{i-1}-1))^\omega.
\]
\end{example}

When $\B=(\beta,\beta,\ldots)$, we recover the usual definition of the quasi-greedy $\beta$-expansion~\cite{Daroczy&Katai:1995,Komornik&Loreti:2007}. In particular, it is easy to check that in this case, if $\DB(1)=\varepsilon_0\cdots \varepsilon_{\ell-1}$ with $\ell\in\N_{\ge 1}$ and $\varepsilon_{\ell-1}> 0$, then the quasi-greedy expansion is purely periodic and $\qDB(1)=(\varepsilon_0\ldots \varepsilon_{\ell-2}(\varepsilon_{\ell-1}-1))^\omega$. For arbitrary Cantor bases, the situation is more complicated and the quasi-greedy expansion can be not periodic. 

\begin{example}
\label{ex:1+sqrt{13}}
Consider the alternate base $\B=\big(\overline{\frac{1+\sqrt{13}}{2},\frac{5+\sqrt{13}}{6}}\big)$. We compute $\DB(1)=201$ and $\DBi{1}(1)=11$. Then $\qDBi{1}(1)=(10)^\omega$ and $\qDB(1)=200\qDBi{1}(1)=200(10)^\omega$.
\end{example}

Moreover, even if the $\B$-expansion is finite, the quasi-greedy $\B$-representation can be infinite not ultimately periodic. Suppose that $\DB(1)$ is finite and that an infinite quasi-greedy is involved during the computation of $\qDB(1)$. Let $n\in \N_{\ge 1}$ be the positive integer such that $\qDBi{n}(1)$ is the  involved infinite expansion. Then $\qDB(1)$ is ultimately periodic if and only if so is $\qDBi{n}(1)$.

\begin{example}
\label{ex:PisotQuadratic}
Consider the Cantor base $\B=(3,\beta,\beta,\beta,\beta,\ldots)$ where $\beta=\sqrt{6}(2+\sqrt{6})$. We get $\DB(1)=3$ and $\DBi{1}(1)=d_{\beta}(1)$ is infinite not ultimately periodic since $\beta$ is a non-Pisot quadratic number~\cite{Bassino:2002}. Therefore, the quasi-greedy expansion $\qDB(1)=2\qDBi{1}(1)$ is not ultimately periodic.
\end{example}

\begin{proposition}
\label{prop:QuasiGreedyRep1}
The quasi-greedy expansion $\qDB(1)$ is a $\B$-representation of $1$.
\end{proposition}

\begin{proof}
It is a straightforward verification.
\end{proof}

\begin{proposition}
\label{prop:qDB1LexGreatest}
If $a$ is an infinite word over $\N$ such that $\val_{\B}(a)< 1$, then $a<_{\lex} \qDB(1)$. Furthermore, $\qDB(1)$ is lexicographically maximal among all infinite $\B$-representations of all real numbers in $[0,1]$.
\end{proposition}

\begin{proof}
If $\DB(1)$ is infinite then the result follows from Corollary~\ref{cor:DB1LexGreatest}. Thus, we suppose that there exists $\ell\in\N_{\ge 1}$ such that $\DB(1)=\varepsilon_0\cdots \varepsilon_{\ell-1}$  and $\varepsilon_{\ell-1}> 0$. 
 
First, let $a\in\N^{\N}$ be such that $\val_{\B}(a)< 1$ and suppose to the contrary that $a\ge_{\lex} \qDB(1)$. By Corollary~\ref{cor:DB1LexGreatest}, $a<_{\lex} \DB(1)$. Then $a_0\cdots a_{\ell-2}=\varepsilon_0\cdots \varepsilon_{\ell-2}$, $a_{\ell-1}=\varepsilon_{\ell-1}-1$ and $\sigma^{\ell}(a)\ge_{\lex}\qDBi{\ell}(1)$. 
Since
\begin{align*}
	\val_{\B}(a)
	&=\sum_{n=0}^{\ell-2}\frac{\varepsilon_n}{\prod_{i=0}^n{\beta_i}}
	+\frac{\varepsilon_{\ell-1}-1}{\prod_{i=0}^{\ell-1}{\beta_i}}	
	+\frac{\val_{\B^{(\ell)}}\big(\sigma^{\ell}(a)\big)}{\prod_{i=0}^{\ell-1}{\beta_i}}\\
	&=1-\frac{1}{\prod_{i=0}^{\ell-1}{\beta_i}}
	\left(
	1-\val_{\B^{(\ell)}}\big(\sigma^{\ell}(a)\big)
	\right),	
\end{align*}
we get that $\val_{\B^{(\ell)}}\big(\sigma^{\ell}(a)\big)<1$. By Corollary~\ref{cor:DB1LexGreatest} again, $\sigma^{\ell}(a)<_{\lex} \DBi{\ell}(1)$. Therefore $\DBi{\ell}(1)$ must be finite and we obtain that $a=\qDB(1)$ by iterating the reasoning. But then $\val_{\B}(a)=1$, a contradiction. 

We now turn to the second part. Suppose that $a\in\N^{\N}$ does not end in $0^\omega$ and is such that $\val_{\B}(a)\le 1$. Our aim is to show that $a\le_{\lex}\qDB(1)$. We know from Corollary~\ref{cor:DB1LexGreatest} that $a\le_{\lex}\DB(1)$. Now, suppose to the contrary that $a>_{\lex}\qDB(1)$. Then $a_0\cdots a_{\ell-2}=\varepsilon_0\cdots \varepsilon_{\ell-2}$, $a_{\ell-1}=\varepsilon_{\ell-1}-1$, and $\sigma^{\ell}(a) >_{\lex} \qDBi{\ell}(1)$. As in the first part of the proof, we obtain that $\val_{\B^{(\ell)}}(\sigma^{\ell}(a))\le 1$ and that $\DBi{\ell}(1)$ must be finite. By iterating the reasoning, we obtain that $a=\qDB(1)$, a contradiction. 
\end{proof}

\section{Admissible sequences}
\label{sec:AdmissibleSequences}

In \cite{Parry:1960}, Parry characterized those infinite words over $\N$ that belong to $D_\beta$.
Such infinite words are sometimes called \emph{$\beta$-admissible sequences}. Analogously, infinite word in $D_{\B}$ are said to be a \emph{$\B$-admissible sequence}. In this section, we generalize Parry's theorem to Cantor bases.

\begin{lemma}
\label{lem:Parry1}
Let $a$ be an infinite word over $\N$ and for each $n\in\N$, let $b^{(n)}$ be a $\B^{(n)}$-representation of $1$. Suppose that for all $n\in\N$, $\sigma^n(a)\le_{\lex} b^{(n)}$. Then for all $k,\ell,m,n\in\N$ with $\ell\ge 1$, the following implication holds:
\begin{equation}
\label{eq:lemParry1}
	a_{k}\cdots a_{k+\ell-1}
	<_{\lex}b^{(n)}_{m}\cdots b^{(n)}_{m+\ell-1} 
	\implies \val_{\B^{(k)}}(a_{k}\cdots a_{k+\ell-1})
				\le \val_{\B^{(k)}}(b^{(n)}_{m}\cdots b^{(n)}_{m+\ell-1}).
\end{equation}
Consequently, for all $k,m,n\in\N$, the following implication holds:
\begin{equation}
\label{eq:lemParry2}
	\sigma^k(a)<_{\lex}\sigma^m(b^{(n)})
	\implies \val_{\B^{(k)}}(\sigma^k(a))
				\le \val_{\B^{(k)}}(\sigma^m(b^{(n)})).
\end{equation}
\end{lemma}

\begin{proof}
Proceed by induction on $\ell$. The base case $\ell=1$ is clear. Let $\ell\ge 2$ and suppose that for all $\ell'<\ell$ and all $k,m,n\in\N$, the implication~\eqref{eq:lemParry1} is true. Now let $k,m,n\in\N$ and suppose that $a_{k}\cdots a_{k+\ell-1}<_{\lex}b^{(n)}_{m}\cdots b^{(n)}_{m+\ell-1}$. Two cases are possible.

Case 1: $a_{k}=b^{(n)}_{m}$. Then $a_{k+1}\cdots a_{k+\ell-1}<_{\lex} b^{(n)}_{m+1}\cdots b^{(n)}_{m+\ell-1}$ and by induction hypothesis, we obtain that
$\val_{\B^{(k+1)}}(a_{k+1}\cdots a_{k+\ell-1})\le \val_{\B^{(k+1)}}(b^{(n)}_{m+1}\cdots b^{(n)}_{m+\ell-1})$. Therefore
\begin{align*}
	\val_{\B^{(k)}}(a_{k}\cdots a_{k+\ell-1})
	& = \frac{a_{k}}{\beta_{k}} 
		+ \frac{\val_{\B^{(k+1)}}(a_{k+1}\cdots a_{k+\ell-1})}{\beta_{k}} \\
	& \le \frac{b^{(n)}_{m}}{\beta_{k}} 
		+ \frac{\val_{\B^{(k+1)}}(b^{(n)}_{m+1}\cdots b^{(n)}_{m+\ell-1})}{\beta_{k}} \\
	& = \val_{\B^{(k)}}(b^{(n)}_{m}\cdots b^{(n)}_{m+\ell-1}). 
\end{align*}

Case 2: $a_{k}<b^{(n)}_{m}$. Since $\sigma^{k+1}(a) \le_{\lex} b^{(k+1)}$ by hypothesis, we have 
\[
	a_{k+1}\cdots a_{k+\ell-1} \le_{\lex} b_0^{(k+1)}\cdots b_{\ell-2}^{(k+1)}.
\] 
%Emilie: C'est ici qu'on a besoin d'avoir l'hypo de récurrence vraie pour tout n!!
By induction hypothesis,
\[
	\val_{\B^{(k+1)}}(a_{k+1}\cdots a_{k+\ell-1}) 
	\le \val_{\B^{(k+1)}}(b_0^{(k+1)}\cdots b_{\ell-2}^{(k+1)})
\le 1.
\] 
Then
\begin{align*}
	\val_{\B^{(k)}}(a_{k}\cdots a_{k+\ell-1})
	& = \frac{a_{k}}{\beta_{k}} 
		+ \frac{\val_{\B^{(k+1)}}(a_{k+1}\cdots a_{k+\ell-1})}{\beta_{k}} \\
	& \le \frac{b^{(n)}_{m}-1}{\beta_{k}} 
		+ \frac{\val_{\B^{(k+1)}}(b_0^{(k+1)}\cdots b_{\ell-2}^{(k+1)})}{\beta_{k}} \\
% 	& \le \frac{b^{(n)}_{m}}{\beta_k} \\
 	& \le \val_{\B^{(k)}}(b^{(n)}_{m}\cdots b^{(n)}_{m+\ell-1}). 	
\end{align*}
Thus, the implication~\eqref{eq:lemParry1} is proved. The implication~\eqref{eq:lemParry2} immediately follows.
\end{proof}

\begin{lemma}
\label{lem:Parry2}
Let $a$ be an infinite word over $\N$ and for each $n\in\N$, let $b^{(n)}$ be a $\B^{(n)}$-representation of $1$. Suppose that for all $n\in\N$, $\sigma^n(a) <_{\lex} b^{(n)}$. Then for all $n\in\N$, $\val_{\B^{(n)}}(\sigma^n(a))<1$ unless there exists $\ell \in\N_{\ge 1}$ such that
\begin{itemize}
\item $b^{(n)}=b^{(n)}_0\cdots b^{(n)}_{\ell-1}$ with $b^{(n)}_{\ell-1}>0$
\item $a_na_{n+1}\cdots a_{n+\ell-1}=b^{(n)}_0\cdots b^{(n)}_{\ell-2}(b^{(n)}_{\ell-1}-1)$
\item $\val_{\B^{(n+\ell)}}(\sigma^{n+\ell}(a))=1$
\end{itemize}
in which case $\val_{\B^{(n)}}(\sigma^n(a))=1$.
\end{lemma}
%Dis comme ça, c'est même plus faible que dans l'article de Parry.

\begin{proof}
Let $n\in\N$. By hypothesis, $\sigma^n(a) <_{\lex} b^{(n)}$. So there exists $\ell\in\N_{\ge 1}$ such that $a_n \cdots a_{n+\ell-2} = b^{(n)}_0\cdots b^{(n)}_{\ell-2}$ and $a_{n+\ell-1}<b^{(n)}_{\ell-1}$. By hypothesis, we also have $\sigma^{n+\ell}(a)<_{\lex}b^{(n+\ell)}$. We get from Lemma~\ref{lem:Parry1} that 
\[
	\val_{\B^{(n+\ell)}}(\sigma^{n+\ell}(a))
	\le \val_{\B^{(n+\ell)}}(b^{(n+\ell)})=1.
\] 
Then
\begin{align*}
	\val_{\B^{(n)}}(\sigma^n(a)) 
	& = \val_{\B^{(n)}}(a_{n}\cdots a_{n+\ell-2})
		+\frac{a_{n+\ell-1}}{\prod_{i=n}^{n+\ell-1}\beta_i}
		+\frac{\val_{\B^{(n+\ell)}}(\sigma^{n+\ell}(a))}{\prod_{i=n}^{n+\ell-1}\beta_i}\\
	& \le \val_{\B^{(n)}}(b^{(n)}_0\cdots b^{(n)}_{\ell-2})
		+\frac{b^{(n)}_{\ell-1}-1}{\prod_{i=n}^{n+\ell-1}\beta_i}
		+\frac{1}{\prod_{i=n}^{n+\ell-1}\beta_i}\\
	& = \val_{\B^{(n)}}(b^{(n)}_0\cdots b^{(n)}_{\ell-1}) \\
	& \le 1.
\end{align*}
Moreover, the equality holds throughout if and only if $b^{(n)}=b^{(n)}_0\cdots b^{(n)}_{\ell-1}$, $a_{n+\ell-1}=b^{(n)}_{\ell-1}-1$ and $\val_{\B^{(n+\ell)}}(\sigma^{n+\ell}(a))=1$. The conclusion follows.
\end{proof}

The following theorem generalizes Parry's theorem for real bases \cite{Parry:1960}.

\begin{theorem}
\label{thm:Parry}
An infinite word $a$ over $\N$ 
belongs to $D_{\B}$ if and only if for all $n\in\N$, $\sigma^n(a)<_{\lex} \DBi{n}^*(1)$.
\end{theorem}

\begin{proof}
In view of Corollary~\ref{cor:ShiftedGreedy}, it suffices to show that the following two assertions are equivalent.
\begin{enumerate}
\item For all $n\in\N$, $\val_{\B^{(n)}}(\sigma^n(a))<1$. 
\item For all $n\in\N$, $\sigma^n(a)<_{\lex} \qDBi{n}(1)$.
\end{enumerate}
The fact that (1) implies (2) follows from Proposition~\ref{prop:qDB1LexGreatest}. Since any quasi-greedy expansion of $1$ is infinite, we obtain that (2) implies (1) by Proposition~\ref{prop:QuasiGreedyRep1} and Lemma~\ref{lem:Parry2}. 
\end{proof}

\begin{example}
Let $\B=(\overline{3,\varphi,\varphi})$ be the alternate base already studied in Examples~\ref{ex:3phiphi} and~\ref{ex2:3phiphi}. Then $a=210(110)^\omega$ is the $\B$-expansion of some $x\in(0,1)$. In fact, since $\qDBi{0}(1)=(210)^\omega$, $\qDBi{1}(1)=(102)^\omega$ and $\qDBi{2}(1)=1(110)^\omega$, by Theorem~\ref{thm:Parry}, there exists $x\in[0,1)$ such that $a=\DB(x)$. We can compute that $a=d_{\B}(\val_{\B}(a))= d_{\B}\big(\frac{19 + 9 \sqrt{5}}{3 (7 + 3 \sqrt{5})}\big)$.
\end{example}

We obtain a corollary characterizing the $\B$-expansions of a real number $x$ in the interval $[0,1]$ among all its $\B$-representations. 

\begin{corollary}
\label{cor:Parry1}
A $\B$-representation $a$ of some real number $x\in[0,1]$ is its $\B$-expansion if and only if for all $n\in\N_{\ge 1}$, $\sigma^n(a)<_{\lex} \qDBi{n}(1)$.
\end{corollary}

\begin{proof}
Let $a\in\N^\N$ be such that $\val_{\B}(a)\in[0,1]$. From Theorem~\ref{thm:Parry}, $\sigma(a)$ belongs to $D_{\B^{(1)}}$ if and only if for all $n\in\N_{\ge 1}$, $\sigma^n(a)<_{\lex} \qDBi{n}(1)$. The conclusion then follows from Proposition~\ref{prop:ShiftedGreedy}.
\end{proof}

\begin{example}
Consider $\B=\big(\overline{\frac{16+5\sqrt{10}}{9},9}\big)$. Then $\DB(1)=\qDB(1)=34(27)^\omega$, $\DBi{1}(1)=90^\omega$ and $\qDBi{1}(1)=834(27)^\omega$. For all $m\in \N_{\ge 1}$, we have $\sigma^{2m}(34(27)^\omega)<_{\lex}\qDB(1)$ and $\sigma^{2m-1}(34(27)^\omega)<_{\lex} \qDBi{1}(1)$ as prescribed by Corollary~\ref{cor:Parry1}.
\end{example}

In comparison with the $\beta$-expansion theory, considering a Cantor base $\B$ and an infinite word $a$ over $\N$, Corollary~\ref{cor:Parry1} does not give a purely combinatorial condition to check whether $a$ is the $\B$-expansion of $1$. We will see in Section~\ref{sec:AlternateBases} that even though an improvement of this result in the context of alternate bases can be proved, a purely combinatorial condition cannot exist.

\section{The $\B$-shift}
\label{sec:BShift}

Let $S_{\B}$ denote the topological closure of $D_{\B}$ with respect to the prefix distance of infinite words: $S_{\B}=\overline{D_{\B}}$.

\begin{proposition}
\label{prop:ParryDS}
An infinite word $a$ over $\N$ 
belongs to $S_{\B}$ if and only if for all $n\in\N$, $\sigma^n(a)\le_{\lex} \DBi{n}^*(1)$.
\end{proposition}

\begin{proof}
Suppose that $a\in S_{\B}$. Then there exists a sequence $(a^{(k)})_{k\in\N}$ of $D_{\B}$ converging to $a$. By Theorem~\ref{thm:Parry}, for all $k,n\in\N$, we have $\sigma^n(a^{(k)})<_{\lex} \qDBi{n}(1)$. By letting $k$ tend to infinity, we get that for all $n\in\N$, $\sigma^n(a)\le_{\lex} \qDBi{n}(1)$. 

Conversely, suppose that for all $n\in\N$, $\sigma^n(a)\le_{\lex} \qDBi{n}(1)$. For each $k\in \N$, let $a^{(k)}=a_0\cdots a_k0^\omega$. Then $\lim\limits_{k\to+\infty}a^{(k)}=a$ and for all $k,n\in\N$, $\sigma^n(a^{(k)})\le_{\lex} \sigma^n(a)\le_{\lex} \qDBi{n}(1)$. Since $\qDBi{n}(1)$ is infinite, for all $k,n\in\N$, $\sigma^n(a^{(k)})<_{\lex} \qDBi{n}(1)$. By Theorem~\ref{thm:Parry}, we deduce that for all $k\in\N$, $a^{(k)}\in D_{\B}$. Therefore $a\in S_{\B}$.
\end{proof}

\begin{proposition}
Let $a,b \in S_{\B}$.
\begin{enumerate}
\item If $a<_{\lex} b$ then $\val_{\B}(a)\leq\val_{\B}(b)$.
\item If $\val_{\B}(a)<\val_{\B}(b)$ then $a<_{\lex} b$.
\end{enumerate}
\end{proposition} 

\begin{proof}
Consider two sequences $(a^{(k)})_{k\in\N}$ and $(b^{(k)})_{k\in\N}$ of $D_{\B}$ such that $\lim_{k\to \infty} a^{(k)}=a$ and $\lim_{k\to \infty} b^{(k)}=b$. Suppose that $a<_{\lex} b$. Then there exists $\ell\in\N_{\ge 1}$ such that $a_0\cdots a_{\ell-1}= b_0 \cdots b_{\ell-1}$ and $a_\ell<b_\ell$. By definition of the prefix distance, there exists $K\in\N$ such that for all $k\ge K$, $a_0^{(k)}\cdots a_{\ell}^{(k)}=a_0\cdots a_{\ell}$ and $b_0^{(k)}\cdots b_{\ell}^{(k)}=b_0\cdots b_{\ell}$. Therefore, for all $k\ge K$, we have $a^{(k)}<_{\lex} b^{(k)}$, and then by Proposition~\ref{pro:Increasing}, $\val_{\B}(a^{(k)})<\val_{\B}(b^{(k)})$. Since the function $\val_{\B}$ is continuous, by letting $k$ tend to infinity, we obtain $\val_{\B}(a)\leq \val_{\B}(b)$. This proves the first item. The second item follows immediately.
\end{proof}

Further, we define 
\[
	\Delta_{\B}=\bigcup_{n\in\N}D_{\B^{(n)}}
	\quad\text{and}\quad
	 \Sigma_{\B}=\overline{\Delta_{\B}}.
\] 

\begin{proposition}
\label{prop:subshift}
The sets $\Delta_{\B}$ and $\Sigma_{\B}$ are both shift-invariant. 
\end{proposition}

\begin{proof}
Let $a$ be an infinite word over $\N$ and $n\in\N$.  It follows from Corollary~\ref{cor:ShiftedGreedy} that if $a\in D_{\B^{(n)}}$ then $\sigma(a)\in D_{\B^{(n+1)}}$. Then, it is easily seen that if $a\in S_{\B^{(n)}}$ then $\sigma(a)\in S_{\B^{(n+1)}}$. 
\end{proof}

Recall some definitions of symbolic dynamics. Let $A$ be a finite alphabet. A subset of $A^\N$ is a \emph{subshift} of $A^{\N}$ if it is shift-invariant and closed with respect to the topology induced by the prefix distance. In view of Proposition~\ref{prop:subshift}, the subset $\Sigma_{\B}$ of $A_{\B}^{\N}$ is a subshift, which we call the ${\B}$-\emph{shift}. For a subset $L$ of $A^\N$, we let $\Fac(L)$ (resp.\ $\Pref(L)$) denote the set of all finite factors (resp.\ prefixes) of all elements in $L$.

\begin{proposition}
\label{prop:Fac}
We have $\Fac(D_{\B})=\Fac(\Delta_{\B})=\Fac(\Sigma_{\B})$.
\end{proposition}

\begin{proof}
By definition, $\Fac(D_{\B})\subseteq\Fac(\Delta_{\B})=\Fac(\Sigma_{\B})$. Let us show that $\Fac(D_{\B})\supseteq\Fac(\Delta_{\B})$. Let $f\in \Fac(\Delta_{\B})$. There exist $n\in\N$ and $a\in D_{\B^{(n)}}$ such that $f\in\Fac(a)$. It follows from Corollary~\ref{cor:ShiftedGreedy} that $0^na$ belongs to $D_{\B}$. Therefore, $f\in\Fac(D_{\B})$. 
\end{proof}

We define sets of finite words $X_{\B,\ell}$ for $\ell\in\N_{\ge 1}$ as follows. If $\qDB(1)=t_0t_1\cdots$ then we let
\[
	X_{\B,\ell}=\{t_0\cdots t_{\ell-2}s \colon s\in[\![0,t_{\ell-1}-1]\!]\}.
\]
Note that $X_{\B,\ell}$ is empty if and only if $t_{\ell-1}=0$.

\begin{proposition}
\label{prop:DX} 
We have
\[
	D_{\B}=\bigcup_{\ell_0\in\N_{\ge 1}} X_{\B,\ell_0}
		\Bigg(\bigcup_{\ell_1\in\N_{\ge 1}} X_{\B^{(\ell_0)},\ell_1}
		\Bigg(\bigcup_{\ell_2\in\N_{\ge 1}} X_{\B^{(\ell_0+\ell_1)},\ell_2}
		\Bigg(
		\quad\cdots\quad
		\Bigg)
		\Bigg)
		\Bigg).
\] 
\end{proposition}

\begin{proof}
For the sake of conciseness, we let $X_{\B}$ denote the right-hand set of the equality. For $n\in\N$, write $\qDBi{n}(1)=t_0^{(n)}t_{1}^{(n)}\cdots$.

Let $a\in D_{\B}$. By Theorem~\ref{thm:Parry}, for all $n\in\N$, $\sigma^n(a)<\qDBi{n}(1)$. In particular, $a<\qDB(1)$. Thus, there exist $\ell_0\in\N_{\ge 1}$ such that $t_{\ell_0-1}^{(0)}>0$ and $s_0\in[\![0,t^{(0)}_{\ell_0-1}-1]\!]$ such that $a=t_0\cdots t_{\ell_0-2}s_0 \sigma^{\ell_0}(a)$. Next, we also have $\sigma^{\ell_0}(a)<\qDBi{\ell_0}(1)$. Then there exist $\ell_1\in\N_{\ge 1}$ such that $t_{\ell_1-1}^{(\ell_0)}>0$ and $s_1\in[\![0,t_{\ell_1-1}^{(\ell_0)}-1]\!]$ such that $\sigma^{\ell_0}(a)=t_0^{(\ell_0)}\cdots t_{\ell_1-2}^{(\ell_0)}s_1 \sigma^{\ell_0+\ell_1}(a)$.  We get that $a\in X_{\B}$ by iterating the process. 

Now, let $a\in X_{\B}$. Then there exists a sequence $(\ell_k)_{k\in\N}$ of $\N_{\ge 1}$ such that $a=u_0u_1u_2\cdots$ where for all $k\in\N$, $u_k \in X_{\B^{(\ell_0+\cdots \ell_{k-1})},\ell_k}$. By Theorem~\ref{thm:Parry}, in order to prove that $a\in D_{\B}$, it suffices to show that for all $n\in\N$, $\sigma^n(a)<_{\lex}\qDBi{n}(1)$. Let thus $n\in\N$. There exist $k\in\N$ and finite words $x$ and $y$ such that $u_k=xy$, $y\ne\varepsilon$ and $\sigma^n(a)=yu_{k+1}u_{k+2}\cdots$. Then $n=\ell_0+\cdots+\ell_{k-1}+|x|$ and $\sigma^n(a)<_{\lex} \sigma^{|x|}\big(\qDBi{\ell_0+\cdots \ell_{k-1}}(1)\big)$. If $x=\varepsilon$ then we obtain $\sigma^n(a)<_{\lex} \qDBi{\ell_0+\cdots \ell_{k-1}}(1)=\qDBi{n}(1)$. Otherwise it follows from Corollary~\ref{cor:Parry1} that $\sigma^{|x|}\big(\DBi{\ell_0+\cdots \ell_{k-1}}(1)\big)
	<_{\lex} \qDBi{\ell_0+\cdots \ell_{k-1}+|x|}(1)
	=\qDBi{n}(1)$, hence we get $\sigma^n(a)
	<_{\lex}\qDBi{n}(1)$ as well.
\end{proof}

\begin{corollary}
We have $D_{\B}=\displaystyle{\bigcup_{\ell\in\N_{\ge 1}} X_{\B,\ell} D_{\B^{(\ell)}}}$.
\end{corollary}

\begin{corollary}
\label{cor:DX}
Any prefix of $\qDB(1)$ belongs to $\Pref(D_{\B})$.
\end{corollary}

\begin{proof}
Write $\qDB(1)=t_0t_1t_2\cdots$ and let $\ell\in\N_{\ge 1}$. Since $\qDB(1)$ is infinite, there exists $k> \ell$ such that $t_{k-1}> 0$. Choose the least such $k$ and let $s\in[\![0,t_{k-1}-1]\!]$. Then $t_0\cdots t_{\ell-1}0^{k-\ell-1}s$ belongs to $X_{\B,k}$. The conclusion follows from Proposition~\ref{prop:DX}.
\end{proof}

\section{Alternate bases}
\label{sec:AlternateBases}

Recall that an alternate base is a periodic Cantor base. The aim of this section is to discuss some results that are specific to these particular Cantor bases. 

We start with a few elementary observations. First, the condition $\prod_{n\in\N}\beta_n=+\infty$ is trivially satisfied in the context of alternate bases since the sequence $(\beta_n)_{n\in\N}$ takes only finitely many values. Then, for an alternate base $\B$ of length $p$, the $\B$-value~\eqref{eq:representationCantor} of an infinite word $a$ over $\R_{\ge 0}$ can be rewritten as
\[
	\val_{\B}(a)
	=\sum_{n\in\N}
	\frac{a_n}{\big(\prod_{i=0}^{p-1}\beta_i\big)^{\lfloor\frac{n}{p}\rfloor} \prod_{i=0}^{n\bmod p}\beta_i}
\]
or as
\begin{equation}
	\label{eq:representationAlternate} 
	\val_{\B}(a)
	=\sum_{m=0}^{+\infty}
	\frac{1}{\big(\prod_{i=0}^{p-1}\beta_i\big)^m}
	\sum_{j=0}^{p-1} 
	\frac{a_{pm+j}}{\prod_{i=0}^{j}\beta_i}.
\end{equation}
Further, the alphabet $A_{\B}$ is finite since $A_{\B}=\{0,\ldots,\max_{i\in\Int}{\floor{\beta_i}}\}$.
Finally, note that a Cantor base of the form $(\beta,\beta,\ldots)$ is an alternate base of length $1$, in which case, as already mentioned in Section~\ref{sec:CantorRealBases}, all definitions introduced so far coincide with those of Rényi~\cite{Renyi:1957} for real bases $\beta$.

In Proposition~\ref{prop:ExistRepresentation}, we gave a characterization of those infinite words $a \in(\R_{\ge 0})^\N$ for which there exists a Cantor base $\B$ such that $\val_{\B}(a)=1$. Here, we are interested in the stronger condition of the existence of an alternate base $\B$ satisfying $\val_{\B}(a)=1$.

\begin{proposition}
\label{pro:ExistAlternateBase}
Let $a$ be an infinite word over $\R_{\ge 0}$ such that $a_n\in O(n^d)$ for some $d\in\N$ and let $p\in\N_{\ge 1}$. There exists an alternate base $\B$ of length $p$ such that $\val_{\B}(a)=1$ if and only if $\sum_{n\in\N}a_n> 1$. If moreover $p\ge 2$, then there exist uncountably many such alternate bases.
\end{proposition}

\begin{proof} 
From Proposition~\ref{prop:ExistRepresentation}, we already know that the condition $\sum_{n\in\N}a_n> 1$ is necessary. Now, suppose that $\sum_{n\in\N}a_n> 1$.
If $p=1$ then the result follows from Lemma~\ref{lem:ExistRepresentation1}. Suppose that $p\ge 2$. Consider any $(p-1)$-tuple $(\beta_1,\ldots,\beta_{p-1})\in (\R_{>1})^{p-1}$. For all $\beta_0>1$, we can write $\val_{\B}(a)=\val_{\beta_0}(c)$ with $\B=(\overline{\beta_0,\beta_1,\ldots,\beta_{p-1}})$ and
\[
	c_m=\frac{1}{\big(\prod_{i=1}^{p-1}\beta_i\big)^m}
	\sum_{j=0}^{p-1} 
	\frac{a_{pm+j}}{\prod_{i=1}^j \beta_i} 
	\quad \text{for all } m\in\N.
\]
Note that $c\in(\R_{\ge 0})^\N$ and that $c_m\in O(m^d)$. By hypothesis, there exists $N\in\N$ such that 
$\sum_{n=0}^{N}a_n>1$. Then
\[
	\sum_{m=0}^{\floor{\frac{N}{p}}}c_m
	> \frac{\sum_{m=0}^{\floor{\frac{N}{p}}} 
 	\sum_{j=0}^{p-1} a_{pm+j}}{\big(\prod_{i=1}^{p-1}\beta_i\big)^{\floor{\frac{N}{p}}+1 }}
	 \ge \frac{\sum_{n=0}^N a_n}{\big(\prod_{i=1}^{p-1}\beta_i\big)^{\floor{\frac{N}{p}}+1 }}.
\]
Therefore, any $(p-1)$-tuple $(\beta_1,\ldots,\beta_{p-1})\in  (\R_{>1})^{p-1}$ satisfying 
\[
	\Bigg(\prod_{i=1}^{p-1}\beta_i\Bigg)^{\floor{\frac{N}{p}}+1}
	\le \;\sum_{n=0}^N a_n
\]
is such that $\sum_{m=0}^{\floor{\frac{N}{p}}}c_m>1$, and hence there exist uncountably many of them. For such a $(p-1)$-tuple, the infinite word $c$ satisfies the hypothesis of Lemma~\ref{lem:ExistRepresentation1}, so there exists $\beta_0>1$ such that $\val_{\B}(a)=\val_{\beta_0}(c)=1$.
\end{proof}

\subsection{The greedy algorithm}
The greedy and the quasi-greedy $\B$-expansions of $1$ enjoy specific properties whenever $\B$ is an alternate base. From now on, we let $\B$ be a fixed alternate base and we let $p$ be its length. 

\begin{proposition}
\label{pro:PurelyPeriodic}
The $\B$-expansion of $1$ is not purely periodic.
\end{proposition}

\begin{proof}
Suppose to the contrary that there exists $q\in\N_{\ge 1}$ such that for all $n\in\N$, $\varepsilon_n=\varepsilon_{n+q}$. By considering $\ell=\lcm(p,q)$, we get that $\B^{(\ell)}=\B$ and for all $n\in\N$, $\varepsilon_n=\varepsilon_{n+\ell}$. Therefore
\[
1	 =\val_{\B}\big(\varepsilon_0\cdots \varepsilon_{\ell-1}\big) 
		+ \frac{1}{\prod_{i=0}^{\ell-1}\beta_i}\\
	 =\val_{\B}\big(\varepsilon_0\cdots \varepsilon_{\ell-2} (\varepsilon_{\ell-1}+1)\big).
\]
Thus $\varepsilon_0\cdots \varepsilon_{\ell-2}(\varepsilon_{\ell-1}+1)$ is a $\B$-representation of $1$ lexicographically greater than $\DB(1)$, which is impossible by Proposition~\ref{pro:GreedyLexGreatest}.
\end{proof}

One might think at first that if for each $i\in\Int$, $\qDBi{i}(1)$ is ultimately periodic, then for $\beta=\prod_{i=0}^{p-1}\beta_i$, $d_\beta^{*}(1)$ must be ultimately periodic as well. This is not the case, as the following example shows. Moreover, the same example shows that the $\B$-expansion of $1$ can be ultimately periodic with a period which is coprime with the length $p$ of $\B$.

\begin{example}
Let $\B=(\overline{\sqrt{6},3,\frac{2+\sqrt{6}}{3}})$. It is easily checked that $\DBi{0}(1)=2(10)^\omega$, $\DBi{1}(1)=3$ and $\DBi{2}(1)=11002$. But the product $\beta=\prod_{i=0}^{p-1}\beta_i=\sqrt{6}(2+\sqrt{6})$ is such that $d_\beta^*(1)$ is not ultimately periodic as explained in Example~\ref{ex:PisotQuadratic}. 
\end{example}

\begin{proposition}
\label{pro:NeverReach}
The quasi-greedy expansion $\qDB(1)$ is ultimately periodic
if and only if either an ultimately periodic expansion is reached or only finite expansions are involved within the first $p$ recursive calls to the definition of $\qDB(1)$.
\end{proposition}

\begin{proof}
If there exists $n\in\N$ such that the infinite expansion $\qDBi{n}(1)$ is involved in the computation of $\qDB(1)$, then clearly $\qDB(1)$ is ultimately periodic if and only if so is $\qDBi{n}(1)$.

Now, suppose that only finite expansions are involved within $p$ recursive calls to the definition of $\qDB(1)$. Then $\DB(1)$ is finite. Thus, $\DB(1)=\varepsilon_{\B,0}\cdots \varepsilon_{\B,\ell_0-1}$ with $\ell_0\in\N_{\geq1}$ and $\varepsilon_{\B,\ell_0-1}> 0$. Then
\[
	\qDB(1)= \varepsilon_{\B,0}\cdots \varepsilon_{\B,\ell_0-2}
			(\varepsilon_{\B,\ell_0-1}-1)
			\qDBi{i_1}(1)
\] 
where $i_1= \ell_0\bmod p$. By hypothesis, $\DBi{i_1}(1)$ is finite as well. Thus we have $\DBi{i_1}(1)=\varepsilon_{\B^{(i_1)},0}\cdots \varepsilon_{\B^{(i_1)},\ell_1-1}$ with $\ell_1\in\N_{\ge 1}$ and $\varepsilon_{\B^{(i_1)},\ell_1-1}> 0$. Repeating the same argument, we obtain 
\[
	\qDBi{i_1}(1)= \varepsilon_{\B^{(i_1)},0}\cdots \varepsilon_{\B^{(i_1)},\ell_1-2}
					(\varepsilon_{\B^{(i_1)},\ell_1-1}-1)
					\qDBi{i_2}(1)
\] 
where $i_2= \ell_0+\ell_1 \bmod p$. By continuing in the same fashion and by setting $i_0=0$, we obtain two sequences $(\ell_j)_{j\in\Int}$ and $(i_j)_{j\in[\![0,p]\!]}$. Because for all $j\in[\![0,p]\!]$, we have $i_j\in\Int$, there exist $j,k\in[\![0,p]\!]$ such that $j<k$ and $i_j=i_k$. Then $\qDB(1)=xy^\omega$ where 
\[
	x=
	\varepsilon_{\B^{(i_0)},0}\cdots \varepsilon_{\B^{(i_0)},\ell_0-2}
			(\varepsilon_{\B^{(i_0)},\ell_0-1}-1)
	\ \cdots \
	\varepsilon_{\B^{(i_{j-1})},0}\cdots \varepsilon_{\B^{(i_{j-1})},\ell_{j-1}-2}
			(\varepsilon_{\B^{(i_{j-1})},\ell_{j-1}-1}-1)\\
\]
and
\[
	y=
	\varepsilon_{\B^{(i_j)},0}\cdots \varepsilon_{\B^{(i_j)},\ell_j-2}
			(\varepsilon_{\B^{(i_j)},\ell_j-1}-1)
	\ \cdots \
	\varepsilon_{\B^{(i_{k-1})},0}\cdots \varepsilon_{\B^{(i_{k-1})},\ell_{k-1}-2}
			(\varepsilon_{\B^{(i_{k-1})},\ell_{k-1}-1}-1).
\]
\end{proof}

\subsection{Admissible sequences}

The condition given in Corollary~\ref{cor:Parry1} does not allow us to check whether a given $\B$-representation of $1$ is the $\B$-expansion of $1$ without effectively computing the quasi-greedy $\B$-expansion of $1$, and hence the $\B$-expansion of $1$ itself. The following proposition provides us with such a condition in the case of alternate bases, provided that we are given the quasi-greedy $\B^{(i)}$-expansions of $1$ for $i\in[\![1,p-1]\!]$.

\begin{proposition}
\label{prop:Parry2}
A $\B$-representation $a$ of $1$ is the $\B$-expansion of $1$ if and only if for all $m\in\N_{\ge 1}$,
$\sigma^{pm}(a)<_{\lex} a$ and for all $m\in\N$ and $i\in[\![1,p-1]\!]$, $\sigma^{pm+i}(a)<_{\lex} \qDBi{i}(1)$.
\end{proposition}

\begin{proof}
The condition is necessary by Corollary~\ref{cor:Parry1} and since $\qDB(1)\le_{\lex}\DB(1)$. Let us show that the condition is sufficient.

Let $a$ be a $\B$-representation of $1$ such that for all $m\in\N_{\ge 1}$, $\sigma^{pm}(a)<_{\lex} a$ and for all $m\in\N$ and $i\in[\![1,p-1]\!]$, $\sigma^{pm+i}(a)<_{\lex} \qDBi{i}(1)$. By Proposition~\ref{pro:GreedyLexGreatest}, $a\le_{\lex}\DB(1)$. By Theorem~\ref{thm:Parry}, if $a<_{\lex}\qDB(1)$ then $\val_{\B}(a)<1$, which contradicts that $a$ is a $\B$-representation of $1$. Thus, $\qDB(1)\le_{\lex} a\le_{\lex} \DB(1)$. If $\DB(1)$ is infinite, then $a=\DB(1)$ as desired. Now, suppose that $\DB(1)=\varepsilon_0\cdots\varepsilon_{\ell-1}$ with $\ell\in\N_{\ge 1}$ and $\varepsilon_{\ell-1}> 0$. Then $a_0\cdots a_{\ell-2}=\varepsilon_0\cdots\varepsilon_{\ell-2}$ and $a_{\ell-1}\in\{\varepsilon_{\ell-1}-1,\varepsilon_{\ell-1}\}$. Since $\val_{\B}(a)=1$, if $a_{\ell-1}=\varepsilon_{\ell-1}$ then $a=\DB(1)$. Therefore, in order to conclude, it suffices to show that $a_{\ell-1}\ne\varepsilon_{\ell-1}-1$. 

Suppose to the contrary that $a_{\ell-1}=\varepsilon_{\ell-1}-1$. Then $\qDBi{\ell}(1)\le_{\lex} \sigma^{\ell}(a)$. By hypothesis, $\ell \equiv 0\pmod p$. Therefore $	\qDB(1)	\le_{\lex} \sigma^{\ell}(a)\le_{\lex} \DB(1)$. By repeating the same argument, we obtain that $a_{\ell}\cdots a_{2\ell-2}=\varepsilon_0\cdots\varepsilon_{\ell-2}$ and $a_{2\ell-1}\in\{\varepsilon_{\ell-1}-1,\varepsilon_{\ell-1}\}$. Since $\sigma^{\ell}(a)<_{\lex} a$ by hypothesis, we must have $a_{2\ell-1}=\varepsilon_{\ell-1}-1$. By iterating the argument, we obtain that $a=(\varepsilon_0\cdots\varepsilon_{\ell-2}(\varepsilon_{\ell-1}-1))^\omega$, contradicting that $\sigma^{\ell}(a)<_{\lex}a$.
\end{proof}

When $p=1$, Proposition~\ref{prop:Parry2} provides us with the purely combinatorial condition proved by Parry~\cite{Parry:1960} in order to determine whether a given $\B$-representation of $1$ is the $\B$-expansion of $1$. However, when $p\ge 2$, we need to compute the quasi-greedy $\B^{(i)}$-expansions of $1$ for every $i\in[\![1,p-1]\!]$ first. This might lead us to a circular computation, in which case the condition may seem not useful in practice. Indeed, suppose that $p=2$ and that we are provided with a $\B$-representation $a$ of $1$ and a $\B^{(1)}$-representation $b$ of $1$. Then in order to check if $a=\DB(1)$, we need to compute $\qDBi{1}(1)$, and hence $\DBi{1}(1)$ first. But then, in order to check if $b=\DBi{1}(1)$, we need to compute $\qDB(1)$, and hence $\DB(1)$, which brings us back to the initial problem. Nevertheless, this condition can be useful to check if a specific $\B$-representation of $1$ is the $\B$-expansion of $1$. For example, consider a $\B$-representation $a$ of $1$ such that for all $m\in \N_{\ge 1}$, $\sigma^{pm}(a)<_{\lex} a$ and for all $m\in \N$ and $i\in[\![1,p-1]\!]$, $a_{pm+i}<\floor{\beta_i}-1$, then the infinite words $a$ satisfies the hypothesis of Proposition~\ref{prop:Parry2} and $a$ is the $\B$-expansion of $1$.

We have seen that considering an infinite word $a$ over $\N$ and a positive integer $p$, there may exist more than one alternate base $\B$ of length $p$ such that $\val_{\B}(a)=1$. Moreover, among all of these alternate bases, it may be that some are such that $a$ is greedy and others are such that $a$ is not. Thus, a purely combinatorial condition for checking whether a $\B$-representation is greedy cannot exist.

\begin{example}
Consider $a=2(10)^\omega$. Then $\val_{\Aalpha}(a)=\val_{\B}(a)=1$ for both $\Aalpha=(\overline{1+\varphi,2})$ and $\B=(\overline{\frac{31}{10},\frac{420}{341}})$. It can be checked that $d_{\Aalpha}(1)=a$ and $\DB(1)\ne a$. 
\end{example}

Furthermore, an infinite word $a$ over $\N$ can be greedy for more than one alternate base. 

\begin{example}
\label{ex:110}
The infinite word $110^\omega$ is the expansion of $1$ with respect to the three alternate bases $(\overline{\varphi,\varphi})$, $(\overline{\frac{5+\sqrt{13}}{6},\frac{1+\sqrt{13}}{2}})$ and $(\overline{1.7,\frac{1}{0.7}})$.
\end{example}

At the opposite, it may happen that an infinite word $a$ is a $\B$-representation of $1$ for different alternate bases $\B$ but that none of these are such that $a$ is greedy. As an illustration, by Proposition~\ref{pro:PurelyPeriodic}, for all purely periodic infinite words $a$ over $\N$, all alternate bases $\B$ such that $\val_{\B}(a)=1$ are such that $a$ is not the $\B$-expansion of $1$.

\begin{example}
The infinite word $(10)^{\omega}$ is a representation of $1$ with respect to the three alternate bases considered in Example~\ref{ex:110}. However, the infinite words $(10)^{\omega}$ is purely periodic therefore, by Proposition~\ref{pro:PurelyPeriodic}, it is not the expansion of $1$ in any alternate base.
\end{example}

\subsection{The $\B$-shift}

We define sets of finite words $Y_{\B,h}$ for $h\in\Int$ as follows. If $\qDB(1)=t_0t_1\cdots$ then we let
\[
	Y_{\B,h}=\{t_0\cdots t_{\ell-2}s \colon 
	\ell\in\N_{\ge 1},\
	\ell\bmod p= h,\
	t_{\ell-1}>0,\
	s\in[\![0,t_{\ell-1}-1]\!]
	\}.
\]
Note that $Y_{\B,h}$ is empty if and only if for all $\ell\in\N_{\ge 1}$ such that $\ell\bmod p= h$, $t_{\ell-1}=0$. So, unlike the sets $X_{\B,h}$ defined in Section~\ref{sec:BShift}, the sets $Y_{\B,h}$ can be infinite. More precisely, $Y_{\B,h}$ is infinite if and only if there exists infinitely many $\ell\in\N_{\ge 1}$ such that $\ell\bmod p= h$ and $t_{\ell-1}>0$. 

\begin{proposition}
\label{prop:DY} 
We have 
\[
	D_{\B}=\bigcup_{h_0=0}^{p-1} Y_{\B,h_0}
		\Bigg(\bigcup_{h_1=0}^{p-1} Y_{\B^{(h_0)},h_1}
		\Bigg(\bigcup_{h_2=0}^{p-1} Y_{\B^{(h_0+h_1)},h_2}
		\Bigg(
		\quad\cdots\quad
		\Bigg)
		\Bigg)
		\Bigg).
\]
\end{proposition}

\begin{proof}
It is easily seen that for all $h\in\Int$, 
\[
	\bigcup_{h=0}^{p-1} Y_{\B,h}
	=\bigcup_{\ell\in\N_{\ge 1}} X_{\B,\ell}.
\]
The conclusion follows from Proposition~\ref{prop:DX}.
\end{proof}

\begin{corollary}
We have $D_{\B}=\displaystyle{\bigcup_{h=0}^{p-1} Y_{\B,h} D_{\B^{(h)}}}$.
\end{corollary}

In the case where all $\qDBi{i}(1)$ are ultimately periodic, we define an automaton $\A_{\B}$ over the finite alphabet $A_{\B}$. Let $\qDBi{i}(1)=t_0^{(i)}\cdots t_{m_i-1}^{(i)} \big(t_{m_i}^{(i)}\cdots t_{m_i+n_i-1}^{(i)}\big)^\omega$. The set of states is 
\[
	Q=\big\{q_{i,j,k}\colon 
		i,j\in\Int,\ k\in[\![0,m_i+n_i-1]\!]\big\}.
\]
The set $I$ of initial states and the set $F$ of final states are defined as
\[
	I=\big\{q_{i,i,0}\colon i\in\Int\big\}
	\quad\text{and}\quad F=Q.
\] 
The (partial) transition function $\delta\colon Q\times A_{\B}\to Q$ of the automaton $\A_{\B}$ is defined as follows. For each $i,j\in\Int$ and each $k\in[\![0,m_i+n_i-1]\!]$, 
we have 
\[
	\delta(q_{i,j,k},t_k^{(i)})=
	\begin{cases}
		q_{i,(j+1)\bmod p,k+1} & \text{ if }k\ne m_i+n_i-1\\
		q_{i,(j+1)\bmod p,m_i} & \text{ else}
	\end{cases}
\]
and for all $s\in[\![0,t_k^{(i)}-1]\!]$, we have
\[
	\delta(q_{i,j,k},s)=q_{(j+1)\bmod p,(j+1)\bmod p,0}.
\]

\begin{example} 
Let $\B=(\overline{\varphi^2, 3+\sqrt{5}})$. Then $\DBi{0}(1)=2(30)^\omega$ and $\DBi{1}(1)=5(03)^\omega$. The corresponding automaton $\mathcal{A}_{\B}$ is depicted in Figure~\ref{fig:Automaton-230-503}.
\begin{figure}[htb]
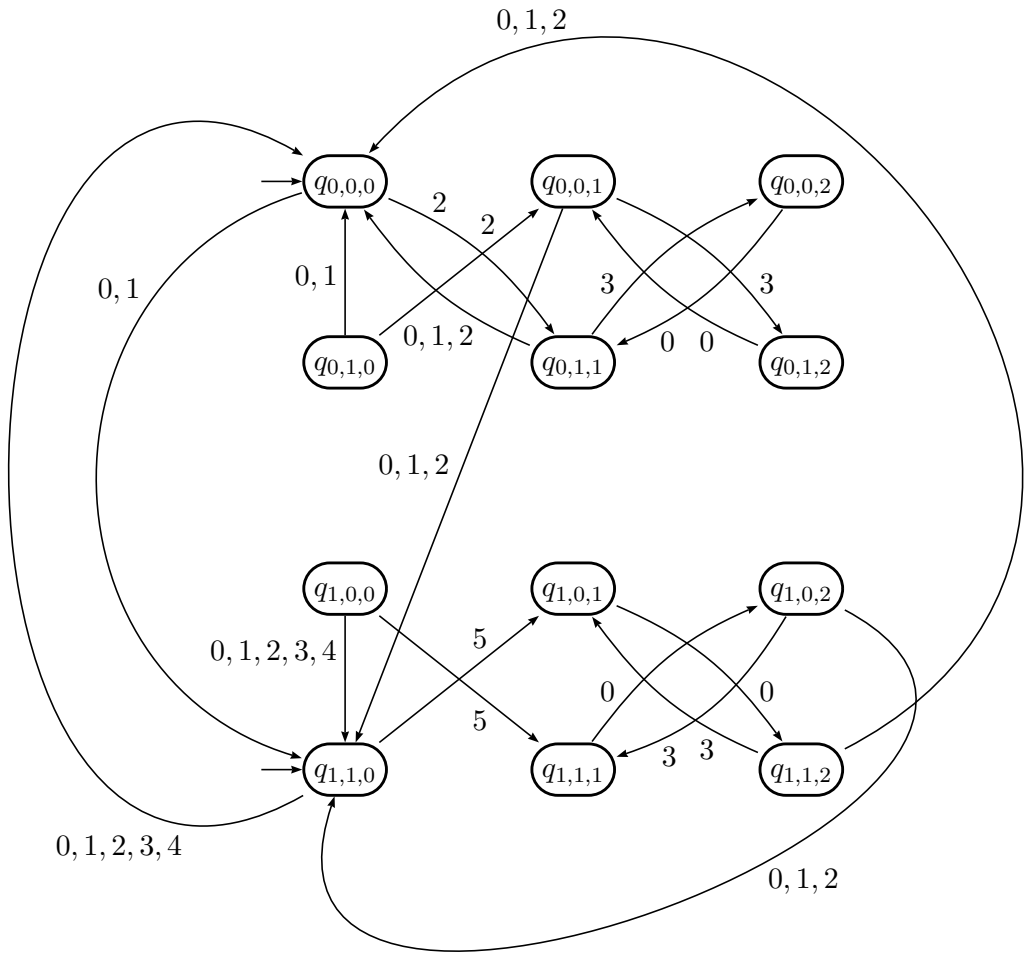

\centering
\VCDraw{\begin{VCPicture}{(0,-4)(10,17)}
% states
 \LargeState
 \StateVar[q_{0,0,0}]{(0,13)}{1}
 \StateVar[q_{0,0,1}]{(5,13)}{2} 
 \StateVar[q_{0,0,2}]{(10,13)}{3} 
 \StateVar[q_{0,1,0}]{(0,9)}{4}
 \StateVar[q_{0,1,1}]{(5,9)}{5} 
 \StateVar[q_{0,1,2}]{(10,9)}{6} 
 \StateVar[q_{1,0,0}]{(0,4)}{7}
 \StateVar[q_{1,0,1}]{(5,4)}{8} 
 \StateVar[q_{1,0,2}]{(10,4)}{9} 
 \StateVar[q_{1,1,0}]{(0,0)}{10}
 \StateVar[q_{1,1,1}]{(5,0)}{11} 
 \StateVar[q_{1,1,2}]{(10,0)}{12} 
% initial--final
\Initial{1}
\Initial{10}
% transitions 
\VArcL[.2]{arcangle=15}{1}{5}{2} 
\ChgLArcCurvature{1}
\VArcR[0.25]{arcangle=-73}{1}{10}{0,1} 

\ChgLArcCurvature{0.8}
\VArcL[.8]{arcangle=15}{2}{6}{3} 
\EdgeR[.5]{2}{10}{0,1,2}

\VArcL[.8]{arcangle=15}{3}{5}{0}

\EdgeL{4}{1}{0,1}
\EdgeL[.75]{4}{2}{2}

\VArcL[.3]{arcangle=15}{5}{1}{0,1,2} 
\VArcL[.2]{arcangle=15}{5}{3}{3}

\VArcL[.2]{arcangle=15}{6}{2}{0}

\EdgeR[.3]{7}{10}{0,1,2,3,4}
\EdgeR[.7]{7}{11}{5} 

\VArcL[.8]{arcangle=15}{8}{12}{0} 

\VArcL[.8]{arcangle=20}{9}{11}{3} 
\ChgLArcCurvature{1.4}
\VArcL{arcangle=130}{9}{10}{0,1,2} 

\EdgeL[.7]{10}{8}{5} 
\VArcL[.15]{arcangle=120}{10}{1}{0,1,2,3,4} 

\ChgLArcCurvature{0.8}
\VArcL[.2]{arcangle=15}{11}{9}{0} 

\VArcL[.2]{arcangle=15}{12}{8}{3} 
\ChgLArcCurvature{1.4}
\VArcR[.85]{arcangle=-100}{12}{1}{0,1,2} 
\end{VCPicture}}
\caption{The automaton $\mathcal{A}_{(\overline{\varphi^2, 3+\sqrt{5}})}$.}
\label{fig:Automaton-230-503}
\end{figure}
By removing the non-accessible states, we obtain the automaton of Figure~\ref{fig:Automaton-230-503-accessible}.
\begin{figure}[htb]
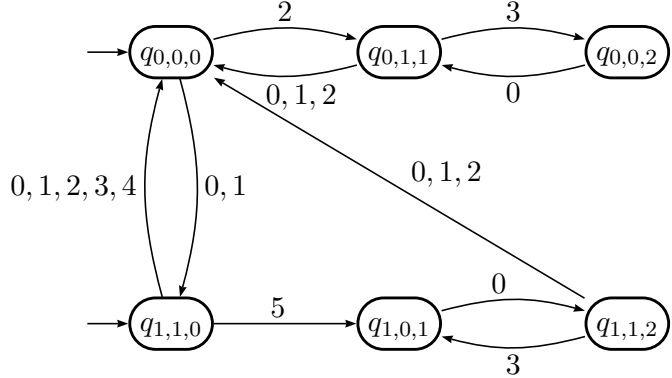

\centering
\VCDraw{\begin{VCPicture}{(0,-1)(10,8)}
% states
 \LargeState
 \StateVar[q_{0,0,0}]{(0,6)}{1}
 \StateVar[q_{0,0,2}]{(10,6)}{3} 
 \StateVar[q_{0,1,1}]{(5,6)}{5} 
 \StateVar[q_{1,0,1}]{(5,0)}{8} 
 \StateVar[q_{1,1,0}]{(0,0)}{10}
 \StateVar[q_{1,1,2}]{(10,0)}{12} 
% initial--final
\Initial{1}
\Initial{10}
% transitions 
\VArcL[.5]{arcangle=15}{1}{5}{2} 
\VArcL[.5]{arcangle=15}{1}{10}{0,1} 

\VArcL[.5]{arcangle=15}{3}{5}{0}

\VArcL[.4]{arcangle=15}{5}{1}{0,1,2} 
\VArcL[.5]{arcangle=15}{5}{3}{3}

\VArcL[.4]{arcangle=15}{8}{12}{0} 

\EdgeL{10}{8}{5} 
\VArcL[.5]{arcangle=15}{10}{1}{0,1,2,3,4} 

\VArcL[.5]{arcangle=15}{12}{8}{3} 
\ChgLArcCurvature{1.4}
\EdgeR{12}{1}{0,1,2} 
\end{VCPicture}}
\caption{An accessible automaton accepting $\Fac(\Sigma_{(\overline{\varphi^2, 3+\sqrt{5}})})$.}
\label{fig:Automaton-230-503-accessible}
\end{figure}
\end{example}

The following result extends a result of Bertrand-Mathis for real bases~\cite{Bertrand-Mathis:1986}.
Recall that a subshift $S$ of $A^{\N}$ is called \emph{sofic} if the language $\Fac(S)\subseteq A^*$ is accepted by a finite automaton. 
%A factor $f\in A^+$ is \emph{avoided} by a subshift $S$ if $f\notin \Fac(S)$. %Clearly, if a factor is contained in a factor avoided by $S$, then it is also avoided by $S$. Therefore, we can talk about the \emph{minimal set of avoided factors} of a subshift $S$. %A subshift $S$ is said to be of \emph{finite type} (resp.\ \emph{sofic} if its minimal set of forbidden factors is finite (resp.\ regular). A subshift $S$ is said to be of \emph{finite type} (resp.\ \emph{sofic} if its minimal set of forbidden factors is finite (resp.\ regular).

\begin{theorem}
The $\B$-shift $\Sigma_{\B}$ is sofic if and only if for all $i\in \Int$, $\qDBi{i}(1)$ is ultimately periodic.
\end{theorem}

\begin{proof} 
Suppose that for all $i\in \Int$, $\qDBi{i}(1)$ is ultimately periodic. We show that the automaton $\A_{\B}$ accepts the language $\Fac(\Sigma_{\B})$. From Propositions~\ref{prop:subshift} and~\ref{prop:Fac}, we obtain that 
\begin{equation}
\label{eq:UnionPref}
	\Fac(\Sigma_{\B})=\Pref(\Delta_{\B})=\bigcup_{i=0}^{p-1}\Pref(D_{\B^{(i)}}).
\end{equation}
Therefore, it suffices to show that for each $i\in\Int$, the language accepted from the initial state $q_{i,i,0}$ is precisely $\Pref(D_{\B^{(i)}})$. Let thus $i\in\Int$. 

First, consider a word $w$ accepted from $q_{i,i,0}$. By Corollary~\ref{cor:DX}, if $w$ is a prefix of $\qDBi{i}(1)$ then $w\in\Pref(D_{\B^{(i)}})$. Otherwise, by construction of $\A_{\B}$, $w$ starts with some $u\in Y_{\B^{(i)},h_0}$ where $h_0=|u|\bmod p$. Moreover, the state reached after reading $u$ from $q_{i,i,0}$ is $q_{j,j,0}$ where $j=(i+h_0)\bmod p$. We obtain that $w\in \Pref(D_{\B^{(i)}})$ by iterating the reasoning and by using Proposition~\ref{prop:DY}. 

Conversely, let $w\in \Pref(D_{\B^{(i)}})$. By Proposition~\ref{prop:DY}, we know that there exists $\ell\in\N$ and $h_0,\ldots,h_\ell\in\Int$ such that $w=u_0\cdots u_{\ell-1}x$ with $u_k\in Y_{\B^{(i+h_0+\cdots h_{k-1})},h_k}$ for all $k\in[\![0,\ell-1]\!]$ and $x$ is a (possibly empty) prefix of $\qDBi{i_\ell}(1)$ where $i_\ell = (i+h_0+\cdots+h_{\ell-1}) \bmod p$. By construction of $\A_{\B}$, by reading $u_0$ from the state $q_{i,0}^{(i)}$, we reach the state $q_{i_1,i_1,0}$ where $i_1= (i+h_0) \bmod p$. Then, by reading $u_1$ from the latter state, we reach the state $q_{i_2,i_2,0}$ where $i_2= (i+h_0+h_1) \bmod p$. By iterating the argument, after reading $u_0\cdots u_{\ell-1}$, we end up in the state $q_{i_\ell,i_\ell,0}$. Since $x$ is a prefix of $\qDBi{i_\ell}(1)$, it is possible to read $x$ from the state $q_{i_\ell,i_\ell,0}$ in $\A_{\B}$. Since all states of $\A_{\B}$ are final, we obtain that $w$ is accepted from $q_{i,i,0}$.

We turn to the necessary condition. Let 
\[
	\qDBi{i}(1)=t_0^{(i)}t_1^{(i)}\cdots\quad \text{ for every } i\in\Int.
\] 
Suppose that $j\in\Int$ is such that $\qDBi{j}(1)$ is not ultimately periodic. Our aim is to find an infinite sequence  $(w^{(m)})_{m\in\N}$ of finite words over $A_{\B}$ such that for all distinct $m,n\in\N$, the words $w^{(m)}$ and $w^{(n)}$ are not right-congruent with respect to $\Fac(\Sigma_{\B})$. Recall that words $x$ and $y$ are not right-congruent with respect to a language $L$ if $x^{-1}L\ne y^{-1}L$, i.e.\ if there exists some word $z$ such that either $xz\in L$ and $yz\notin L$, or $xz\notin L$ and $yz\in L$. If we succeed then we will know that the number of right-congruence classes is infinite and we will be able to conclude that $\Fac(\Sigma_{\B})$ is not accepted by a finite automaton.

We define a partition $(G_1,\ldots,G_q)$ of $\Int$ as follows. Let $r=\Card\{\qDBi{i}(1)\colon i\in\Int\}$ and let $i_1,\ldots,i_r\in\Int$ be such that $\qDBi{i_1}(1),\ldots,\qDBi{i_r}(1)$ are pairwise distinct. Without loss of generality, we can suppose that $\qDBi{i_1}(1)>_{\lex}\cdots>_{\lex}\qDBi{i_r}(1)$. Let $q\in[\![1,r]\!]$ be the unique index such that $\qDBi{i_q}(1)=\qDBi{j}(1)$. We set 
\[
	G_s=\{i\in\Int\colon \qDBi{i}(1)=\qDBi{i_s}(1)\}\quad \text{for }s\in[\![1,q-1]\!] 
\]
and 
\[
	G_q=\{i\in\Int\colon \qDBi{i}(1)\le \qDBi{j}(1)\}.
\] 
For each $s\in[\![1,q-1]\!]$, we write $G_s=\{i_{s,1},\ldots,i_{s,\alpha_s}\}$ where $i_{s,1}<\ldots<i_{s,\alpha_s}$ and we use the convention that $i_{s,\alpha_s+1}=i_{s+1,1}$ for $s\le q-2$ and $i_{q-1,\alpha_{q-1}+1}=j$. Moreover, we let $g\in\N_{\ge 1}$ be such that for all $i,i'\in \Int$ such that $\qDBi{i}(1)\ne \qDBi{i'}(1)$, the length-$g$ prefixes of $\qDBi{i}(1)$ and $\qDBi{i'}(1)$ are distinct. Then, for $s\in[\![1,q-1]\!] $, we define $C_s$ to be the least $c\in\N_{\ge 1}$ such that $t^{(i_s)}_{g-1+c}> 0$. Finally, let $N\in\N_{\ge 1}$ be such that $pN\ge \max\{g,C_1,\ldots,C_{q-1}\}$. 

For all $m\in\N$, consider 
\[
	w^{(m)}
	=\left(
	\prod_{s=1}^{q-1} 
	\prod_{k=1}^{\alpha_s}
	t_0^{(i_s)}\cdots t_{g-1}^{(i_s)} 
	0^{p(2N+1)-g+i_{s,k+1}-i_{s,k}} 
	\right) 
	t_0^{(j)}\cdots t_{m-1}^{(j)}.
\] 
For all $m\in\N$, $s\in [\![1,q-1]\!]$ and $k\in [\![1,\alpha_s]\!]$, the factor $t_0^{(i_s)}\cdots t_{g-1}^{(i_s)}0^{p(2N+1)-g+i_{s,k+1}-i_{s,k}}$ has length $p(2N+1)+i_{s,k+1}-i_{s,k}$, and hence occurs at a position congruent to $i_{s,k}-i_{1,1}$ modulo $p$ in $w^{(m)}$. Similarly, for all $m\in\N$, the factor $t_0^{(j)}\cdots t_{m-1}^{(j)}$ occurs at a position congruent to $j-i_{1,1}$ modulo $p$ in $w^{(m)}$. These observations will be crucial in what follows. The situation is illustrated in Figure~\ref{fig:wm}.\begin{figure}[htb]
\begin{tikzpicture}
\draw (-1,0.1) node[above]{$w^{(m)}=$};
\draw (0,0) rectangle (2,0.6); 
\draw (1,0) node[above]{$w_{1,1}$};
\draw (0,0) node[]{$\bullet$};
\draw (0,0) node[below]{${\scriptstyle 0}$};
\draw (0.2,-0.9) node[]{$\cdots$};
\draw (2.2,0) rectangle (4,0.6);
\draw (3.1,0) node[above]{$w_{1,2}$};
\draw (2.2,0) node[]{$\bullet$};
\draw (2.2,0) node[below]{${\scriptstyle i_{1,2}-i_{1,1}}$};
\draw (5.5,0) node[above]{$\cdots$};
\draw (7,0) rectangle (8.7,0.6);
\draw (7.85,0) node[above]{$w_{1,\alpha_1}$};
\draw (7,0) node[]{$\bullet$};
\draw (7,0) node[below]{${\scriptstyle i_{1,\alpha_1}-i_{1,1}}$};
\draw (0.2,-2.9) node[]{$\cdots$};
\draw (0,-2) rectangle (2.3,-1.4); 
\draw (1.15,-2) node[above]{$w_{s,1}$};
\draw (0,-2) node[]{$\bullet$};
\draw (0,-2) node[below]{${\scriptstyle i_{s,1}-i_{1,1}}$};
\draw (2.5,-2) rectangle (4.5,-1.4);
\draw (3.5,-2) node[above]{$w_{s,2}$};
\draw (2.5,-2) node[]{$\bullet$};
\draw (2.5,-2) node[below]{${\scriptstyle i_{s,2}-i_{1,1}}$};
\draw (5.95,-2) node[above]{$\cdots$};
\draw (7.4,-2) rectangle (9.5,-1.4);
\draw (8.45,-2) node[above]{$w_{s,\alpha_s}$};
\draw (7.4,-2) node[]{$\bullet$};
\draw (7.4,-2) node[below]{${\scriptstyle i_{s,\alpha_s}-i_{1,1}}$};
\draw (0,-4) rectangle (2.2,-3.4); 
\draw (1.1,-4) node[above]{$w_{q-1,1}$};
\draw (0,-4) node[]{$\bullet$};
\draw (0,-4) node[below]{${\scriptstyle i_{q-1,1}-i_{1,1}}$};
\draw (2.4,-4) rectangle (4.2,-3.4);
\draw (3.3,-4) node[above]{$w_{q-1,2}$};
\draw (2.4,-4) node[]{$\bullet$};
\draw (2.4,-4) node[below]{${\scriptstyle i_{q-1,2}-i_{1,1}}$};
\draw (5.35,-4) node[above]{$\cdots$};
\draw (6.5,-4) rectangle (8.8,-3.4);
\draw (7.65,-4) node[above]{$w_{q-1,\alpha_{q-1}}$};
\draw (6.5,-4) node[]{$\bullet$};
\draw (6.5,-4) node[below]{${\scriptstyle i_{q-1,1}-i_{1,1}}$};
\draw (0,-5.25) rectangle (2.5,-4.65);
\draw (1.25,-5.4) node[above]{$t_0^{(j)}\cdots t_{m-1}^{(j)}$};
\draw (0,-5.25) node[]{$\bullet$};
\draw (0,-5.25) node[below]{${\scriptstyle j-i_{1,1}}$};
\end{tikzpicture}
\caption{Positions modulo $p$ of the occurrences of the factors $w_{k,s}$ and $t_0^{(j)}\cdots t_{m-1}^{(j)}$ in $w^{(m)}$, where $w_{k,s}=t_0^{(i_s)}\cdots t_{g-1}^{(i_s)} 0^{p(2N+1)-g+i_{s,k+1}-i_{s,k}}$.}
\label{fig:wm}
\end{figure}
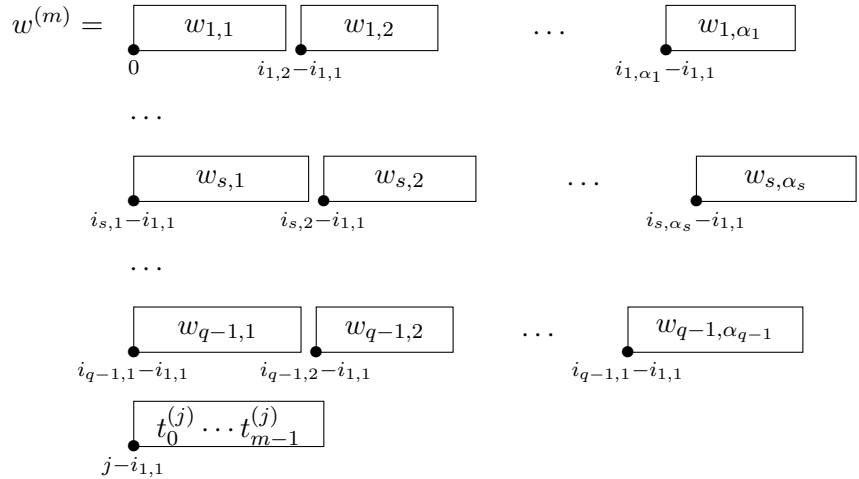

Now, let $m,n\in\N$ be distinct. Since $\qDBi{j}(1)$ is not ultimately periodic, $\sigma^m \big(\qDBi{j}(1)\big)\ne \sigma^n\big(\qDBi{j}(1)\big)$. Thus, there exists $\ell\in\N_{\ge 1}$ such that $t_m^{(j)}\cdots t_{m+\ell-2}^{(j)}=t_n^{(j)}\cdots t_{n+\ell-2}^{(j)}$ and $t_{m+\ell-1}^{(j)}\ne t_{n+\ell-1}^{(j)}$. Without loss of generality, we suppose that $t_{m+\ell-1}^{(j)}>t_{n+\ell-1}^{(j)}$. Let $z=t_m^{(j)}\cdots t_{m+\ell-1}^{(j)}$. Our aim is to show that $w^{(m)}z \in \Fac(\Sigma_{\B})$ and $w^{(n)}z\notin \Fac(\Sigma_{\B})$.

In order to obtain that $w^{(m)}z \in \Fac(\Sigma_{\B})$, we show that $w^{(m)}z \in \Pref(D_{\B^{(i_{1,1})}})$. First, for all $s\in[\![1,q-1]\! ]$ and $k\in[\![1,\alpha_s]\!]$, $t_0^{(i_s)}\cdots t_{g-1}^{(i_s)}0^{C_s}\in Y_{\B^{(i_{s,k})},(g+C_s) \bmod p}$. Second, for all $i\in\Int$, $0\in Y_{\B^{(i)},1}$. Third, by Corollary~\ref{cor:DX}, for all $h\in\Int$,  $t_0^{(j)}\cdots t_{m-1}^{(j)}z\in\Pref(Y_{\B^{(j)},h})$. The conclusion follows from Proposition~\ref{prop:DY}.

In view of \eqref{eq:UnionPref}, in order to prove that $w^{(n)}z\notin \Fac(\Sigma_{\B})$, it suffices to show that for all  $i\in\Int$, $w^{(n)}z\notin \Pref(D_{\B^{(i)}})$. Proceed by contradiction and let $i\in \Int$ and $w\in D_{\B^{(i)}}$ such that $w^{(n)}z$ is a prefix of $w$. By Theorem~\ref{thm:Parry}, for all $s\in [\![1,q]\!]$, the factor $t_0^{(i_s)}\cdots t_{g-1}^{(i_s)} 0^{C_s}$  occurs at a position $e$ in $w$ such that $(i+e)\bmod p$ belongs to $G_1\cup\cdots\cup G_s$.
For $s=1$, we obtain that for all $k\in[\![1,\alpha_1]\!]$, $(i+i_{1,k}-i_{1,1})\bmod p\in G_1$, and hence that
\[
	G_1=\{(i+i_{1,1}-i_{1,1})\bmod p,\ldots,(i+i_{1,\alpha_1}-i_{1,1})\bmod p\}.
\]
For $s=2$, we get that for all $k\in[\![1,\alpha_2]\!]$, 
$(i+i_{2,k}-i_{1,1})\bmod p\in G_1\cup G_2$. If $(i+i_{2,k}-i_{1,1})\bmod p\in G_1$ for some $k\in[\![1,\alpha_2]\!]$, then there exists $k'\in[\![1,\alpha_1]\!]$ such that $(i+i_{2,k}-i_{1,1})\bmod p=(i+i_{1,k'}-i_{1,1})\bmod p$, hence such that $i_{2,k}=i_{1,k'}$, which is impossible since $G_1$ and $G_2$ are pairwise disjoint. It follows that
\[
	G_2=\{(i+i_{2,1}-i_{1,1})\bmod p,\ldots,(i+i_{2,\alpha_2}-i_{1,1})\bmod p\}.
\]
By iterating the reasoning, we obtain that 
\[
	G_s=\{(i+i_{s,1}-i_{1,1})\bmod p,\ldots,(i+i_{s,\alpha_s}-i_{1,1})\bmod p\}
	\quad\text{for all } s\in[\![1,q-1]\!].
\]
We finally get that $(i+j-i_{1,1})\bmod p$ belongs to $G_q$. Then $\qDBi{(i+j-i_{1,1})\bmod p}(1)\le_{\lex}\qDBi{j}(1)$. Let $r$ be the position where the factor $t_0^{(j)}\cdots t_{n-1}^{(j)}$ occurs in $w^{(n)}$, and hence also in $w$ since $w^{(n)}z$ is a prefix of $w$. We have seen that $r\equiv j-i_{1,1}\pmod p$. Since $w\in D_{\B^{(i)}}$, it follows from Theorem~\ref{thm:Parry} that 
\[
	\sigma^r(w)
	<_{\lex}\qDBi{i+r}(1)
	=\qDBi{(i+j-i_{1,1})\bmod p}(1)
	\le_{\lex}\qDBi{j}(1).
\] 
We have thus reached a contradiction since the factor $t_0^{(j)}\cdots t_{n-1}^{(j)}z$ is lexicographically greater than the length-($n+\ell$) prefix of $\qDBi{j}(1)$.
\end{proof}

Note that, in the classical case $p=1$, the previous proof is much shorter since $\Fac(\Sigma_{\beta})=\Pref(D_{\beta})$, and hence we can directly deduce that the words $t_0^{(j)}\cdots t_{m-1}^{(j)}$ and $t_0^{(j)}\cdots t_{n-1}^{(j)}$ (where in fact, $j=0$) are not right-congruent with respect to $\Fac(\Sigma_{\beta})$.

A subshift $S$ of $A^{\N}$ is said to be of \emph{finite type} if its minimal set of forbidden factors is finite. For $p=1$, it is well known that the $\beta$-shift is of finite type if and only if $d_\beta(1)$ is finite \cite{Bertrand-Mathis:1986}. However, this result does not generalize to $p\ge 2$ as is illustrated by the following example.

\begin{example}
Consider the alternate base $\B=(\overline{\frac{1+\sqrt{13}}{2},\frac{5+\sqrt{13}}{6}})$ of Example~\ref{ex:1+sqrt{13}}. Then $\qDBi{0}(1)=200(10)^\omega$ and $\qDBi{1}(1)=(10)^\omega$. We see that all words in $2(00)^{*}2$ are factors avoided by $\Sigma_{\B}$, so the $\B$-shift $\Sigma_{\B}$ is not of finite type.
\end{example}

%{\color{red} Parler de la $\B$-tranformation $T_{\B}$ pour introduire l'article avec Karma et faire le liens entre les mots admissibles et les codes des trajectoires décrite par cette transformation.}

%{\color{red} Autres questions? Exemples et contre-exemples?}

%{\color{red} Question ouverte : Qu'est-ce qu'un p-uple "Parry"? (car dans Hollander on sait ce qui généralise)}

%{\color{red} Parler des Cantor ultimement périodiques (dire que l'idée est la même)}

\section{Acknowledgment}
Célia Cisternino is supported by the FNRS Research Fellow grant 1.A.564.19F.

\bibliographystyle{abbrv}
\bibliography{CharlierCisternino2020}

\end{document}